\documentclass[10pt, twoside, a4paper]{article}
\usepackage[utf8]{inputenc}
\usepackage{amsmath, amssymb, amsthm}            
\usepackage{amstext, amsfonts, a4}
\usepackage[english]{babel}
\usepackage{stmaryrd}
\usepackage{bbm}
\usepackage{graphicx}
\usepackage{caption} 

\usepackage[margin=3cm]{geometry}

\theoremstyle{plain}
	\newtheorem{Theo}{Theorem}[section]
	
	\newtheorem{Prop}[Theo]{Proposition}    
	
	\newtheorem{Lemm}[Theo]{Lemma}          
	\newtheorem{Coro}[Theo]{Corollary}

\theoremstyle{definition}
	\newtheorem{Defi}[Theo]{Definition}

\theoremstyle{remark}
	\newtheorem{Rema}[Theo]{Remark}

\DeclareMathOperator{\conv}{conv}
\DeclareMathOperator{\supess}{sup\: ess\;}
\DeclareMathOperator{\infess}{inf\: ess\;}
\DeclareMathOperator{\meas}{meas}
\DeclareMathOperator{\supp}{supp}
\DeclareMathOperator{\dist}{dist}

             \def\NN{{\mathbb N}}    \def\RR{{\mathbb R}}        \def\ZZ{{\mathbb Z}} 

     \def\cB{{\mathcal B}}  \def\cH{{\mathcal H}}   \def\cC{{\mathcal C}}          \def\cE{{\mathcal E}}            

   \def\mfM{{\mathfrak M}}

\begin{document}
\title{Tamped functions: A rearrangement in dimension 1.}
\author{Ludovic Godard-Cadillac\footnote{Sorbonne Universit\'e, Laboratoire Jacques-Louis Lions, 4, Place Jussieu, 75005 Paris, France. E-mail: {\tt ludovic.godardcadillac@ljll.math.upmc.fr}}}
\date{\today}

\maketitle
\begin{abstract}
We define a new rearrangement, called rearrangement by tamping, for non-negative measurable functions defined on $\RR_+$. This rearrangement has many properties in common with the well-known Schwarz non-increasing rearrangement such as the Pólya–Szegő inequality. Contrary to the Schwarz rearrangement, the tamping also preserves the homogeneous Dirichlet boundary condition of a function.
\end{abstract}

\tableofcontents

\section*{Introduction}
~\indent 
The Schwarz non-increasing rearrangement is a powerful tool to establish symmetry properties of solutions to
variational problems \cite{AlvinoTrombettiLions}\cite{Brezis2}. In this paper, we are interested in situations where the Schwarz rearrangement cannot be applied
as such due to boundary constraints. Consider for example the simple minimization problem
$$\min\bigg\{ \int_0^1|\varphi'(x)|^2\,\mathrm{d}x,\quad\varphi\in H^1_0\big(]0,1[,\RR\big),\quad\int_0^1 g(x)\,\varphi^2(x)\, \mathrm{d}x = 1\bigg\},$$
where  $g: [0,1]\to \RR$ is positive and
non-decreasing. Due to the constraint $\varphi(0)=0$, the unique minimizer $\varphi^\star$ cannot be non-increasing. Since the weight $g(x)$ is decreasing, it is favorable to shift as much of the mass of  $\varphi^\star$ towards the origin. However, this process is balanced by the Dirichlet energy which would otherwise blow-up, since $\varphi(0)=0$, if $\varphi^\star$ is too much concentrated near the origin. For that reason, it appears natural to expect that $\varphi^\star$ is unimodal, meaning that there exists $s^\star$ in $(0,1)$ such that $\varphi^\star$ is non-decreasing on $[0,s^\star]$ and non-increasing on $[s^\star,1]$. Such a result still holds with a more general constraint of the form $\int_0^1F(\varphi(x),x)\mathrm{d}x=1,$ where $F:\RR^2\to\RR_+$ is non-decreasing for each variables.
As a matter of fact, any minimizer of
the optimization problem is invariant by the rearrangement by tamping as a consequence of its four main properties. First, the rearrangement by tamping satisfies the P\'olya-Szeg\"o inequality, as stated in Theorem
\ref{decrease_W_1_p} and the Schwarz rearrangement inequality, stated at Property \ref{Schwarz_propr}. It also maps the set of measurable non-negative functions into the set of unimodal functions. Finally, in addition to these three properties that it shares with the Schwarz non-increasing rearrangement, the tamping preserves the Dirichlet boundary condition of non-negative functions of $W^{1,p}_0(0,1)$ or $W^{1,p}_0(\RR_+)$. The main object of this work consists in defining the rearrangement by tamping and establishing its main properties.

On the other hand, this rearangement does \emph{not} satisfy the Hardy-Littlewood nor the Riesz rearrangement inequalities.     
The continuity properties of the rearrangement by tamping are also somewhat weaker than their
equivalent for the Schwarz non-increasing rearrangement \cite{Coron}\cite{AlmgrenLieb}. Still, Theorem \ref{continuity_theorem_1D}
states some form of continuity of the tamping process in $L^p(\RR_+)$ ($1\leq p < +\infty$). In our
construction, this convergence property is essential for extending the properties established for the tamping, such as the Pólya–Szegő inequality, from elementary
step functions to arbitrary functions in $L^p(\RR_+).$ Compactness results for the rearrangement by
tamping are also proved (Theorem \ref{compactness_result} and its corollaries).

\section{Presentation of the problem}\label{section1}
In this section we provide a short survey on rearrangements with an emphasis on the properties that we are interested in. In the last subsection we explain that the classical Schwarz non-increasing rearrangement fails to preserve the Dirichlet boundary conditions and we present what are the properties that we require for our rearrangement. 
We denote by $\mfM_+(\Omega)$ the set of non-negative measurable functions defined on a domain $\Omega$, and by 
$$\{\varphi\geq\nu\}:=\{x\in \Omega\;:\;\varphi(x)\geq\nu\}$$ the \emph{superlevel sets} of a function $\varphi \in \mfM_+(\Omega).$ The sets $\{\varphi>\nu\}$ and $\{\varphi=\nu\}$ are defined similarly.
\subsection{Layer-cake representation}
Let $\Omega$ be an open subset of $\RR^d$.
We recall here the \emph{layer-cake representation} of a non-negative measurable function \cite{LiebLoss},
\begin{equation}\label{layercake}\varphi(x)\;=\;\int_0^{\varphi(x)}\!\!\mathrm{d}\nu\;=\;\int_0^{+\infty}\mathbbm{1}_{\{\varphi\geq\nu\}}(x)\;\mathrm{d}\nu,\end{equation}
where $\mathbbm{1}_A$ refers to the indicator function of the set $A$. This layer-cake representation implies that for $f\in\cC^1(\RR_+,\RR_+)$ non-decreasing such that $f(0)=0$,
\begin{equation}\label{layercake2}\int_\Omega f\circ\varphi(x)\,\mathrm{d}x\;=\;\int_0^{+\infty}f'(\nu)\,\meas\left(\{\varphi\geq\nu\}\right)\,\mathrm{d}\nu.\end{equation}
In this equality, $f$ can actually be chosen in $BV(\RR_+)$ and in this case its derivative is a measure. In particular,
\begin{equation}\label{layercake_Lp}\big\|\varphi\big\|_{L^p}^p\;=\;p\int_0^{+\infty}\nu^{p-1}\,\meas\left(\{\varphi\geq\nu\}\right)\,\mathrm{d}\nu.\end{equation}

\subsection{The Schwarz non-increasing rearrangement}
\begin{Defi}\label{define_rearrangement} \indent
Let $\varphi$ be a non-negative measurable function. The function $\psi$ is a \emph{rearrangement} of the function $\varphi$ if and only if for almost every $\nu$ we have
$$\meas\left(\{\varphi\geq\nu\}\right)=\meas\left(\{\psi\geq\nu\}\right).$$
\end{Defi}
As a consequence of the layer-cake representation \eqref{layercake2}, we have the following property.
\begin{Prop}\label{semilinearities0} 
Let $\varphi$ be a non-negative measurable function and let $\psi$ be a rearrangement of $\varphi$. The following equality holds (whenever these quantities are finite). 
\begin{equation}\label{semilinearities}\int_\RR \Big(f\circ\varphi\Big)(s)\,\mathrm{d}s\;=\;\int_\RR \Big(f\circ\psi\Big)(s)\,\mathrm{d}s,\end{equation}
for any BV-function $f:\RR\to\RR$. In particular,
\begin{equation}\label{semilinearities2}\|\varphi\|_{L^p}\;=\;\|\psi\|_{L^p}.\end{equation}
\end{Prop}

\begin{Defi}\label{Schwarz_rearrangement} \indent
Let $\varphi$ be a non-negative measurable function defined on $\RR_+$. The \emph{Schwarz non-increasing rearrangement of $\varphi$}, that we note $\varphi^\ast$, is defined as being the only rearrangement of $\varphi$ that is a non-increasing function on $\RR_+$. Its superlevel sets are
$$\{\varphi^\ast\geq\nu\}\;=\;[0\,;\,\meas\{\varphi\geq\nu\}].$$
\end{Defi}
The Schwarz non-increasing rearrangement can be interpreted as being the rearrangement which ``\emph{shoves all the mass of the function until it reaches the origin}''. The idea that this rearrangement ``\emph{moves the mass of the function to the left}'' is embedded in the following important inequality.
\begin{Prop}[Schwarz rearrangement inequality]\label{Schwarz_propr}
$$\forall\;x\in\RR_+,\qquad\meas\Big([0,x]\cap\{\varphi\geq\nu\}\Big)\;\leq\;\meas\Big([0,x]\cap\{\varphi^\ast\geq\nu\}\Big).$$
\end{Prop}
As a consequence of this inequality, the Schwarz non-increasing rearrangement also satisfies several other inequalities exploiting this idea that it shoves the mass down to the origin. For instance we can obtain inequalities for weighted $L^p$-norms when the weight is non-increasing. More generally we have the following two properties.
\begin{Prop}\label{semilinearities_completed2} 
Let $f:\RR\to\RR$ be a non-negative mesurable function and let $g:\RR_+\to\RR$ non-increasing. Let $\varphi\in\mfM_+(\RR_+)$. Then,
\begin{equation}
\forall\;x\geq0,\qquad\int_0^x\Big(f\circ\varphi\Big)(s)\,g(s)\,\mathrm{d}s\;\leq\;\int_0^x \Big(f\circ\varphi^\ast\Big)(s)\,g(s)\,\mathrm{d}s.
\end{equation}
\end{Prop}
\begin{proof}
Using twice the layer-cake representation we obtain
\begin{align*}
\int_0^x\Big(f\circ\varphi\Big)(s)\,g(s)\,\mathrm{d}s&=\int_0^x\int_0^{+\infty}\mathbbm{1}_{\{f\circ\varphi\geq\nu\}}(s)\mathrm{d}\nu\,g(s)\,\mathrm{d}s\\
&=\int_0^x\int_0^{+\infty}\int_0^{+\infty}\mathbbm{1}_{\{f\circ\varphi\geq\nu\}}(s)\;\mathbbm{1}_{\{g\geq\mu\}}(s)\,\mathrm{d}\mu\,\mathrm{d}\nu\,\mathrm{d}s\\
&=\int_0^x\int_0^{+\infty}\int_0^{+\infty}\mathbbm{1}_{\{f\circ\varphi\geq\nu\}\cap\{g\geq\mu\}}(s)\,\mathrm{d}\mu\,\mathrm{d}\nu\,\mathrm{d}s.
\end{align*}
Since $g$ is non-increasing, the set $\{g\geq\mu\}$ is a segment starting at $0$. Noticing that $(f\circ\varphi)^\ast=f\circ\varphi^\ast$ we can conclude using the Schwarz rearrangement inequality.
\end{proof}
\begin{Prop}\label{semilinearities_completed3} 
Let $f:\RR\to\RR_+$ be non-decreasing $BV$-functions such that $\inf f =0$ and let $g:\RR\to\RR$ be non-decreasing. Let $\varphi\in\mfM_+(\RR_+)$. Then,
\begin{equation}\label{Chevalier Danceny}
\forall\;x\geq0,\qquad\int_0^x f\circ\Big(\varphi-g\Big)(s)\,\mathrm{d}s\;\leq\;\int_0^x f\circ\Big(\varphi^\ast-g\Big)(s)\,\mathrm{d}s.
\end{equation}
\end{Prop}
\begin{proof}
Since $f$ is a $BV$-function, its weak derivative $f'$ is a signed measure. Using the layer-cake representation, we have
\begin{equation}\label{Viconte de Valmont}
\int_0^x f\circ\Big(\varphi-g\Big)(s)\,\mathrm{d}s=\int_{0}^{+\infty}\meas\Big(\{\varphi-g\geq\nu\}\cap[0,x]\Big)\,f'(\mathrm{d}\nu).
\end{equation}
We now observe that
\begin{equation}\label{Marquise de Mertreuil}
\Big\{x\in\RR_+:\varphi(x)-g(x)\geq\nu\Big\}=\bigcup_{\mu\geq\inf g}\Big\{x\in\RR_+:\varphi(x)\geq\nu+\mu\Big\}\cap\Big\{x\in\RR_+:g(x)\leq\mu\Big\}.
\end{equation}
Since $g$ is non-decreasing, the set $\{g\leq\mu\}$ is a segment starting at $0$ and then we can conclude using the Schwarz rearrangement inequality that
\begin{equation}
\meas\Big(\{\varphi-g\geq\nu\}\cap[0,x]\Big)\leq\meas\Big(\{\varphi^\ast-g\geq\nu\}\cap[0,x]\Big).
\end{equation}
Since $f$ is a non-decreasing function, the measure $f'$ is actually non-negative and then \eqref{Viconte de Valmont} with \eqref{Marquise de Mertreuil} gives \eqref{Chevalier Danceny}.
\end{proof}
Integral terms of the form $\int_0^xf(\varphi((s))g(s)\,\mathrm{d}s$ like in Proposition \ref{semilinearities_completed2} appear in the variational formulation of many reaction-diffusion equations. In these cases, $\varphi$ models the density of the studied population while $g$ can be understood as the hostility of the environment (see e.g. \cite{BerestyckiLachandRobert}). Integral terms of the form $\int_0^x f(\varphi(s)-g(s))\,\mathrm{d}s$ also appear in the variational formulation of some problems in fluid mechanics (see e.g. \cite{BergerFraenkel}\cite{Norbury}\cite{GravejatSmets}). 

\subsection{Rearrangement inequalities}\label{rearrangement inequalities}
These are the three main rearrangement inequalities for the Schwarz and Steiner rearrangements. The proofs can be found respectively in \cite{HardyLittlewoodPolya}, \cite{PolyaSzego}, \cite{RieszInequality}.
\begin{Theo}[Hardy-Littlewood rearrangement inequality]\label{Hardy_Littlewood} \indent
Let $\varphi$ and $\psi$ be two non-negative measurable functions defined on $\RR_+$. Then,
$$\int_0^\infty\varphi\,\psi\;\leq\;\int_0^\infty\varphi^\ast\,\psi^\ast.$$
\end{Theo}
It must be said that this inequality is not a consequence of the Schwarz rearrangement inequality \ref{Schwarz_propr} nor implies it. This inequality is also true for the Steiner rearrangement.
\begin{Theo}[Pólya–Szegő inequality]\label{Polya_Szego} \indent
If we suppose that $\varphi\in\dot{W}_+^{1,p}$ then so is the function $\varphi^\ast$ and
$$\int_{\RR_+}\big|\nabla(\varphi^\ast)\big|^p\leq\int_{\RR_+}|\nabla\varphi|^p.$$
\end{Theo}
The corresponding inequality also holds in the case of the Steiner rearrangement. The case $p=2$ in the Pólya–Szegő inequality is of particular interest since it involves the energy term associated to a Laplace operator.
\begin{Theo}[Riesz Rearrangement inequality]\label{Riesz} \indent
Let $\varphi,$ $\psi$ and $\chi$ $:\RR\to\RR_+$. Then,
$$\int_{\RR}\int_{\RR}\varphi(x)\,\psi(y)\,\chi(x-y)\,\mathrm{d}x\,\mathrm{d}y\leq\int_{\RR}\int_{\RR}\varphi^\sharp(x)\,\psi^\sharp(y)\,\chi^\sharp(x-y)\,\mathrm{d}x\,\mathrm{d}y.$$
\end{Theo}
This inequality can be considered as being a generalization of the Hardy-Littlewood inequality \ref{Hardy_Littlewood} if we choose, at least formally, the function $\chi$ to be equal to the Dirac mass centered at $0$. An important consequence of this inequality is the fact that for $s\in(0,1)$ we have,
\begin{equation}\big|\varphi^\sharp\big|_{H^s}\leq\;\big|\varphi\big|_{H^s}.\end{equation}
Here the $H^s$ half-norms are defined by
\begin{equation*}
|\varphi|_{H^s}^2\;:=\;\int_\RR\int_\RR\frac{|\varphi(x)-\varphi(y)|^2}{|x-y|^{2(1-s)}}\,\mathrm{d}x\,\mathrm{d}y.
\end{equation*}

\subsection{Limitation of the Schwarz rearrangement, preserving Dirichlet boundary condition}

\subsubsection{The problem of tamping in dimension $1$}\label{problem_tamping_1D}
\label{sec:problem of tamping}
However useful the Schwarz rearrangement may be, it does not work if we want to impose a Dirichlet boundary condition at $0$  because such a condition is in general not respected after a Schwarz rearrangement. In the case of a Dirichlet boundary condition, the solution is expected to be unimodal instead of non-increasing, where by "unimodal" we mean non-decreasing on an interval $[0,s]$ and then non-increasing on $[s,+\infty)$ for a certain $s\in\RR_+$. In short, we want to build a rearrangement that
\begin{itemize}
\item Satisfies a Pólya–Szegő inequality (Theorem \ref{Polya_Szego}),
\item Satisfies a Schwarz rearrangement inequality (Property \ref{Schwarz_propr}),
\item Gives unimodal functions (non-decreasing then non-increasing),
\item Preserves the Dirichlet boundary condition at $0$.
\end{itemize}

The main observation that starts this work is the following intuitive explanation of why the Schwarz rearrangement verifies the Pólya–Szegő inequality. It can be said, roughly speaking, that when we apply the Schwarz rearrangement, the mass of the function is pushed all the way down to the origin and when this movement of mass is done, all \emph{the ``hollows'' of the function are filled up} (we mean that all the strict local minima have disappeared after the Schwarz rearrangement). The idea that the Pólya–Szegő inequality is a consequence of the fact that the ``hollows'' are filled up, is the guiding idea of our construction. The spirit of the rearrangement by tamping is that the mass has to be moved in order to \emph{``fill all the hollows''} but we must not move this mass \emph{``too far''}, because this implies the loss of the Dirichlet boundary condition.

\begin{figure}[h]\centering\centering
\includegraphics[width=12.5cm]{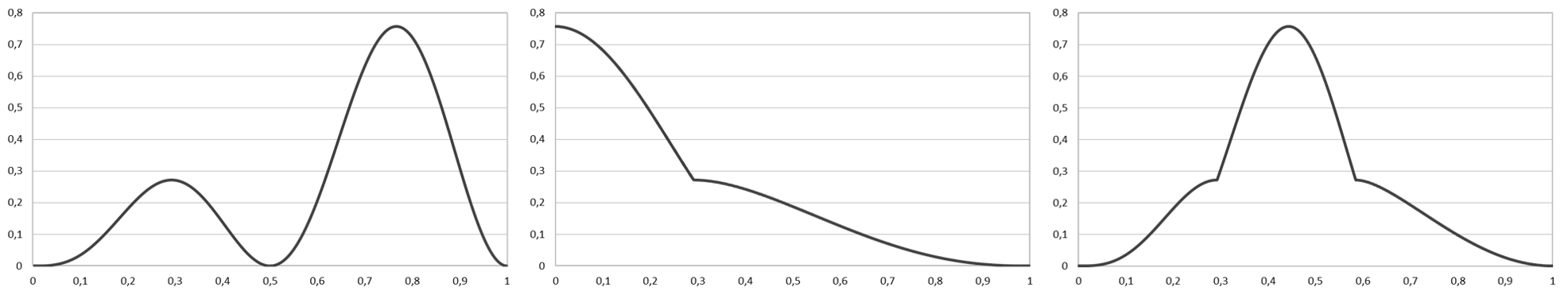}
\caption{\label{xsin2x} From left to right: the graph of the function $x\longmapsto x.\sin^2(2\pi x).\mathbbm{1}_{[0,1]}(x)$, the graph of the Schwarz rearrangement of this function and the graph of the rearrangement by tamping.}
\end{figure} 

As an illustration, Figure \ref{xsin2x} shows the graph of $x\longmapsto x.\sin^2(2\pi x).\mathbbm{1}_{[0,1]}(x)$, the graph of the Schwarz non-increasing rearrangement of this function and the graph of the rearrangement by tamping that we are introducing in this article.

\subsubsection{The problem of tamping in dimension $2$}\label{problem_tamping_2D}
Although our construction is so far limited to functions of one variable, the problem of tamping in dimension $2$ can still be enunciated. It consists in building a rearrangement for non-negative functions defined on $\RR_+\times\RR$ that satisfies a 2D-analogous of the four conditions given at Section \ref{problem_tamping_1D} for the 1D case.
\begin{itemize}
\item The extension of the Pólya–Szegő inequality in dimension two consists in replacing the derivative of the functions by a gradient.
\item For a problem on $\RR_+\!\times\!\RR$, the Schwarz rearrangement inequality becomes
$$\forall\;x\in\RR_+,\quad\meas\Big([0,x]\!\times\!\RR\;\cap\;\{\varphi\geq\nu\}\Big)\;\leq\;\meas\Big([0,x]\!\times\!\RR\;\cap\;\{\varphi^\ast\geq\nu\}\Big),$$
where here \emph{meas} refers to the Lebesgue measure on $\RR^2$.
\item In dimension $2$, the unimodal functions are the functions whose set of local maxima is a connected set (they are just ``\emph{one bump}'').
\item On $\RR_+\!\times\!\RR$, the Dirichlet boundary condition is set on $\{0\}\!\times\!\RR.$
\end{itemize}

It must be said that the result obtained in dimension $1$ for the problem of tamping cannot be immediately extended to the problem of tamping in dimension $2$ because this rearrangement does not show good properties of tensorization.

\section{Definition of the rearrangement by tamping}

\subsection{Definition of the tamping on voxel functions}\label{subsection21}
At first we define the rearrangement by tamping on a special subclass of the piece-wise constant functions called the \emph{voxel functions}. The general case for functions in $\mfM_+(\RR_+)$ will be treated later.
The definition of the tamping rearrangement on voxel functions is given by an algorithm separated into two algorithms. 
We first define the \emph{elementary tamping}. The general algorithm of tamping then consists in an iteration of the elementary tamping algorithm. Before all, we have to define what is a voxel function. 

\subsubsection{The voxel functions}\label{voxel_things}

Let $n\in\NN^\ast$ and let $\Gamma:\{1,\cdots,n\}^2\longrightarrow\{0,1\}$ a function being non-increasing with respect to its second variable, in other words 
\begin{equation}
\forall\; 1 \leq i, j, k \leq n,\ \left.\begin{array}{c}k\leq j\\
\;\Gamma(i,j)=1\end{array}\;\right\}\Longrightarrow\;\Gamma(i,k)=1.
\label{Gamma_defines_function}\end{equation}
For convenience we extend the function $\Gamma$ to $\NN^2$ by setting its values to zero outside of $\{1,\cdots,n\}^2.$ We define the \emph{voxels} (associated to $\Gamma$) as being the sets
\begin{equation}\label{voxel_1D}a_\Gamma(i,j):=\left\{\begin{array}{cl}\left[i-\frac{1}{2},\,i+\frac{1}{2}\right]\times\left[j-1,\,j\right] & \quad\mathrm{if}\;\;\Gamma(i,j)=1,\\ \\
\emptyset & \quad\mathrm{if}\;\;\Gamma(i,j)=0,\end{array} \right. \end{equation}
and then the \emph{voxel pile} 
\begin{equation}
A_\Gamma:=\bigcup_{i,j}\;a_\Gamma(i,j)\quad\subset\RR^2.
\label{hypograph_Gamma}\end{equation}
If we only consider the reunion on $j$ then the obtained set is referred as being the \emph{voxel column} at abscissa $i$ and if on the contrary we consider the reunion on the index $i$ then we get the \emph{voxel line} at ordinate $j$.
Condition \eqref{Gamma_defines_function} implies that the set $A_\Gamma$ is the hypograph of a certain piece-wise constant function $\varphi_\Gamma:\RR\longrightarrow\RR_+$. The set of piece-wise constant functions that can be defined in such a way is noted $\cE_n(\RR)$, and its elements are called \emph{voxel functions}.

\subsubsection{The elementary tamping algorithm}
For $\xi\in\NN$ one defines the function:
\begin{equation}\label{define_eta}
\eta(\xi):=\min\left\{\eta\in\NN^\ast\;:\;\xi\leq\eta\quad\mathrm{and}\quad \varphi(\eta-1)>\varphi(\xi)\geq \varphi(\eta)\right\}.
\end{equation}
Roughly speaking, the \emph{elementary tamping algorithm} consists in moving one step on the left the mountain that is between $\xi$ and $\eta(\xi)$ (see Figure \ref{elementary_tamping_1D}). This $\xi$ will often be referred as being the \emph{``pivot''} of the elementary tamping algorithm.
\begin{Defi} \indent
The \emph{elementary tamping} associated to a given pivot $\xi$ consists in modifying the function $\Gamma$ following the algorithm:
\newline\indent\indent\qquad for $j$ from $\varphi(\xi)+1$ to $n$ do:
\newline\indent\indent\qquad\qquad for $i$ from $\xi$ to $\eta(\xi)-1$ do:
\newline\indent\indent\qquad\qquad\qquad $\Gamma(i,j)\;\longleftarrow\;\Gamma(i+1,j)$
\end{Defi}
\begin{figure}[h]\centering
\includegraphics[width=12cm]{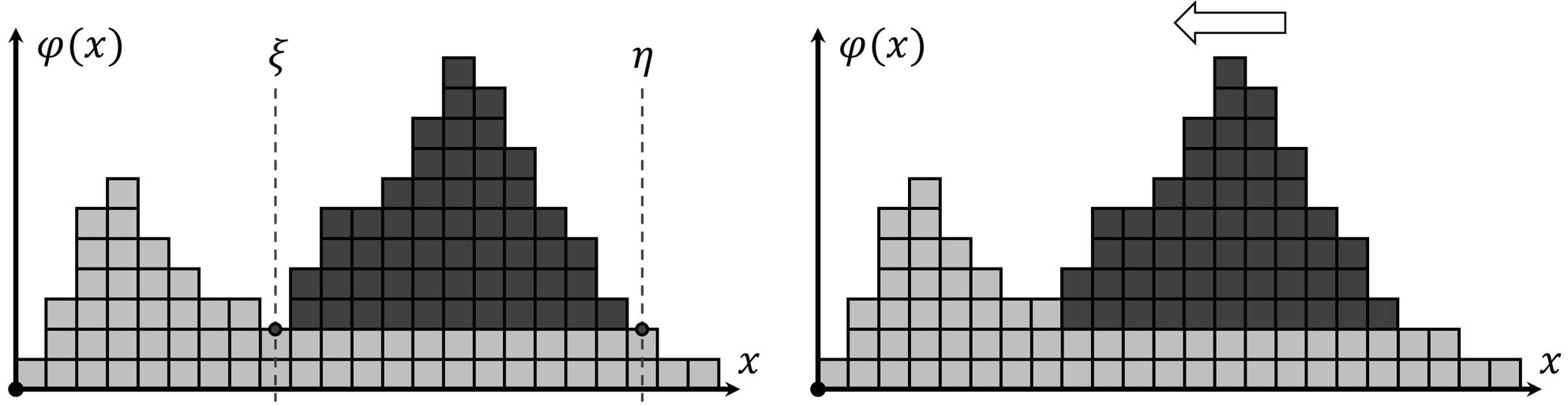}
\caption{\label{elementary_tamping_1D} The elementary tamping algorithm on a voxel function: the red voxels are slid one step to the left.}
\end{figure}
Once this definition is set, it is important to check that the new function $\Gamma$ that we get after the elementary tamping still verifies Property \eqref{Gamma_defines_function}. This is crucial because we want the function $\Gamma$ to still define a function $\varphi_\Gamma$.
\begin{Lemm} \indent\label{defines_a_new_function}
The elementary tamping preserves Property given by \eqref{Gamma_defines_function}.
\end{Lemm}

\begin{proof} \indent
To understand what happens during one step of the algorithm, we consider the hypograph of the function as being a collection of vertical strips of width 1 but of different height. 
\begin{figure}[h]\centering
\includegraphics[width=10cm]{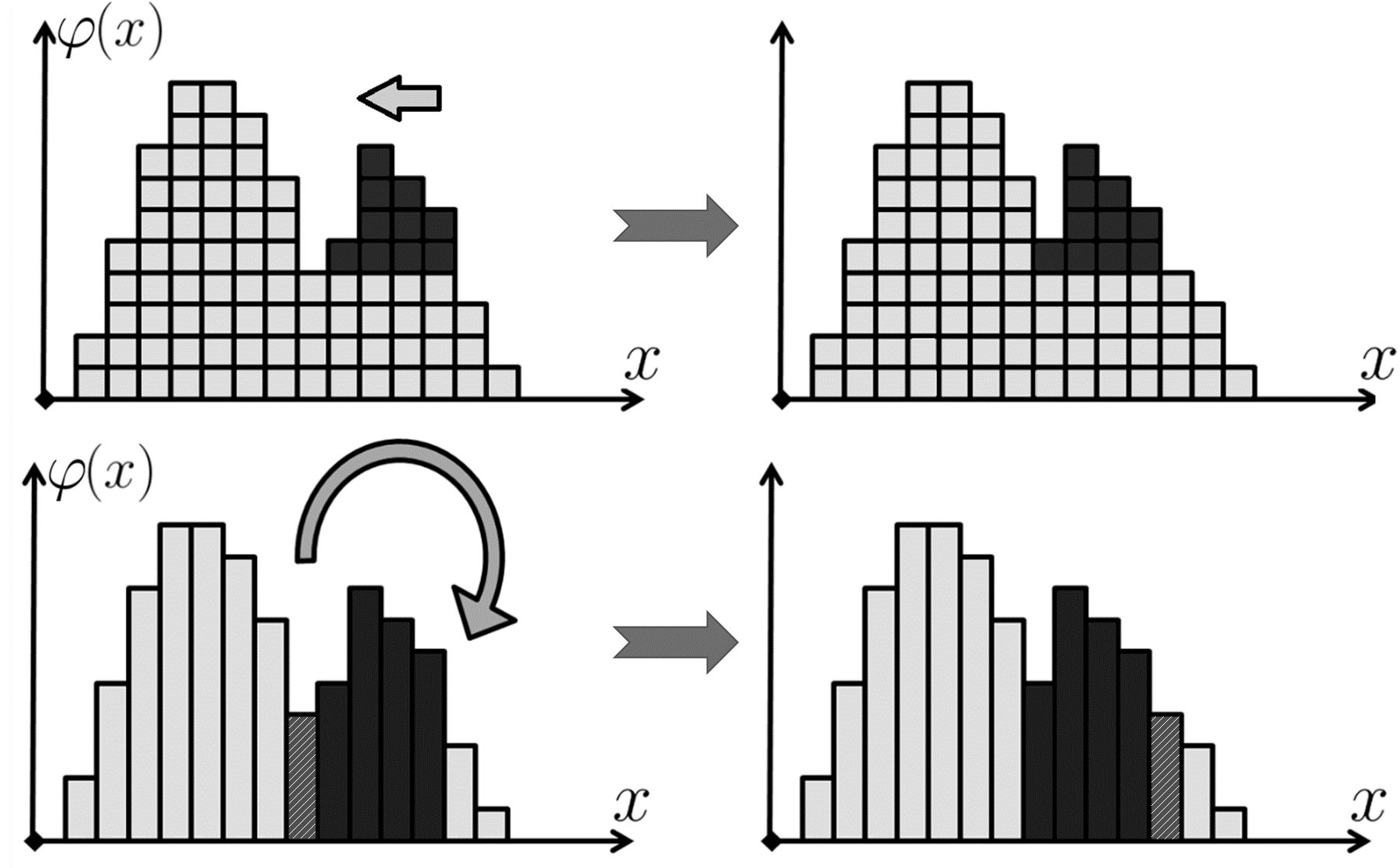}
\caption{\label{tamping_equivalence} The equivalence of the two algorithms.
Top: Lebesguian point of view on the tamping.
Bottom: Riemannian point of view on the tamping.
}
\end{figure}\newline
A step of the algorithm can be reformulated as the following (see Figure \ref{tamping_equivalence}):
\begin{enumerate}
\item The strip at $\xi$ is removed. There is now a hollow (by hollow, we mean a missing strip) at $\xi$.
\item All the strips between $\xi+1$ and $\eta$ are translated one step to the left. Therefore, the hollow is now between $\eta$ and $\eta+1$. 
\item The hollow is filled up with the strip that we removed at Step 1.
\end{enumerate}

We can see that both algorithms are equivalent but with this version it is clear that Property \eqref{Gamma_defines_function} is preserved during the tamping process.
\end{proof}
We have now two equivalent algorithms to define the elementary step of the tamping algorithm. We call the first point of view the \emph{Lebesguian} point of view and the second one is called the \emph{Riemannian} point of view.

\subsubsection{The tamping algorithm}
\begin{Defi}\label{1D_tamping_algorithm} \indent
Once the elementary tamping process is defined, the \emph{tamping process}, also called \emph{rearrangement by tamping}, is defined by the following algorithm:
\begin{enumerate}
\item Define the sets:
\begin{align}\label{define_N}
&\underline{M}:=\left\{x\in\NN\;:\;\varphi(x)<\varphi(x+1)\right\},\\
&\underline{N}:=\left\{x\in\NN\;:\;\exists\;t\in\NN,\; t < x,\;\forall s\in\Big]t+\frac{1}{2},\;x\Big[,
\;\varphi(t)>\varphi(s)\right\}.
\end{align}
\item If the set $\underline{N}\bigcap\underline{M}$ is empty, the algorithm ends.
\item Otherwise, choose a pivot $\xi\in \underline{N}\bigcap\underline{M}$.
\item Apply the elementary tamping associated to this pivot $\xi$.
\item Restart the algorithm at Step 1.
\end{enumerate}
\end{Defi}
The result of this algorithm is illustrated at Figure \ref{tamping_1D}. The set $\underline{N}\bigcap\underline{M}$ must be understood as the set of all the strict local minima of the function $\varphi$. Its definition comes from the fact that the classical definition for local minima makes no sense for piece-wise constant functions although this is the kind of concept we need. Similarly is defined the set $\overline{N}\bigcap\overline{M}$ of the strict local maxima for $\varphi$ with
\begin{align}
&\overline{M}:=\left\{x\in\NN^\ast\;:\;\varphi(x-1)<\varphi(x) \right\},\\
&\overline{N}:=\left\{x\in\NN^\ast\;:\;\exists\;t\in\NN,\; t > x,\; \forall s\in\Big]x,\;t-\frac{1}{2}\Big[,
\;\varphi(t)<\varphi(s) \right\}.\label{define_N_2}
\end{align}
\begin{Lemm} \indent\label{convergence_algorithm}
This algorithm converges in a finite number of iterations.
\end{Lemm}
\begin{proof} \indent
The non-negative integer $\displaystyle N_\Gamma:=\sum_{i,j}\;i\;\Gamma(i,j)$ is decreasing at every iteration.
\end{proof}
Given a boolean function $\Gamma$ and $\varphi_\Gamma\in\cE_n(\RR)$, one notes $\Gamma^\natural$ the boolean function given by the previous algorithm when initialized with the function $\Gamma$. This function is well defined because of the following uniqueness property.
\begin{Lemm} \indent\label{unicity}
The boolean function $\Gamma^\natural$ does not depend on the choice of $\xi$ at every iteration of the algorithm.
\end{Lemm}
\begin{proof} \indent
This fact is a consequence of Proposition \ref{L^p_tamping} that we prove later.\newline
\end{proof}
The tamped-function, also called the rearrangement by tamping of $\varphi_\Gamma$, that we note $\left(\varphi_\Gamma\right)^\natural$, is defined by the natural definition
\begin{equation}\label{definition_tamping}
\left(\varphi_\Gamma\right)^\natural:=\varphi_{\Gamma^\natural}.
\end{equation}
This definition naturally extends to the functions of the set of the dilatations of the voxel functions
\begin{equation}\label{extention_of_E_n}
\widetilde{\cE_n}(\RR_+)\;:=\;\Big\{\lambda.\varphi\Big(\frac{\cdot}{\mu}\Big)\;:\;\varphi\in\cE_n(\RR_+),\;\lambda,\mu>0\Big\},
\end{equation}
which is the set of the voxel functions made with voxels of size $\lambda\times\mu$. We set the definition
\begin{equation}\label{first_extention_of the tamping}
\Big[\lambda.\varphi\Big(\frac{\cdot}{\mu}\Big)\Big]^\natural\;:=\;\lambda.\varphi^\natural\Big(\frac{\cdot}{\mu}\Big).
\end{equation}
We also define
$$\widetilde{\cE}(\RR_+):=\bigcup_{n=1}^\infty\widetilde{\cE_n}(\RR_+).$$
\begin{figure}[h!]\centering
\includegraphics[width=12cm]{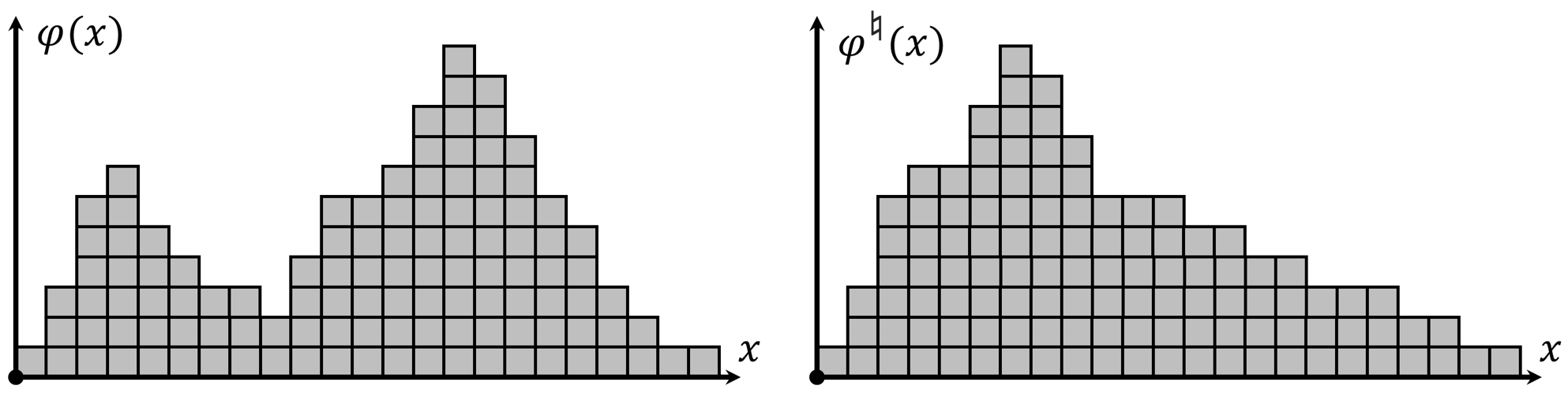}
\caption{\label{tamping_1D} The 1D tamping algorithm on a voxel function.}
\end{figure}

\subsection{Definition of the tamping in $\mfM_+(\RR_+)$.}\label{subsection22}

\subsubsection{Problems of point-wise convergence for the tamping}
The tamping is now defined on voxel functions. It is a well-known result that piece-wise constant functions are a dense subset of the set of measurable functions $\mfM(\RR_+)$ \cite{LiebLoss} for the point-wise convergence almost everywhere and it is always possible to approximate a piece-wise constant function by voxel functions of $\widetilde{\cE_n}(\RR)$. The main objective of this part is to define the tamping for any function in $\mfM_+(\RR_+)$ and the natural way to do it is to pass to the limit in which the size of the voxels tends to zero (and their number tends to $+\infty$). Nevertheless, such an approach cannot work because the tamping has unfortunately very bad continuity properties for the point-wise convergence. For instance if we consider the sequence
$$\varphi_n\;:=\;\mathbbm{1}_{[0,1]}+\mathbbm{1}_{[n,\,n+1]},$$
then we have $\varphi_n$ converging almost everywhere towards $\varphi:=\mathbbm{1}_{[0,1]}$ (which is a fixed point for the tamping) whereas $\varphi_n^\natural$ is converging almost everywhere towards $\mathbbm{1}_{[0,2]}\neq\varphi.$ 

A good way to pass through the difficulties evoked just before is to define the tamping by another approach but that coincides with the previous one on the sets $\widetilde{\cE_n}(\RR)$.
\subsubsection{Definition of the hollows}
To understand well the purpose of what follows, one must keep in mind that the main idea of the tamping is to ``\emph{move the cubes to the left in order to fill all the hollows}''. This key idea is the basis for the generalization of the tamping for all the functions in $\mfM_+(\RR_+)$. The idea behind the notion of \emph{hollows} is to seize the lack of convexity of a given set (see Figure \ref{hollows}, Left). A natural definition for the \emph{hollows} of a set $A$ would be
\begin{equation}\label{fake_define_hollows}\cH(A):=(\conv A)\setminus A.\end{equation}
Nevertheless, such a definition is not well adapted for the manipulation of sets in a context of measure theory and integration with respect to the Lebesgue measure. For instance the convex hull can be completely changed if we add only one point to the set $A$. The objective here is to define a notion of \emph{hollows} that is the analog of the natural idea sketched by \eqref{fake_define_hollows} but defined in such a way that $\cH(A)$ remains unchanged if we modify the set $A$ by a set of measure zero.

\begin{figure}[h]\centering
\includegraphics[width=10cm]{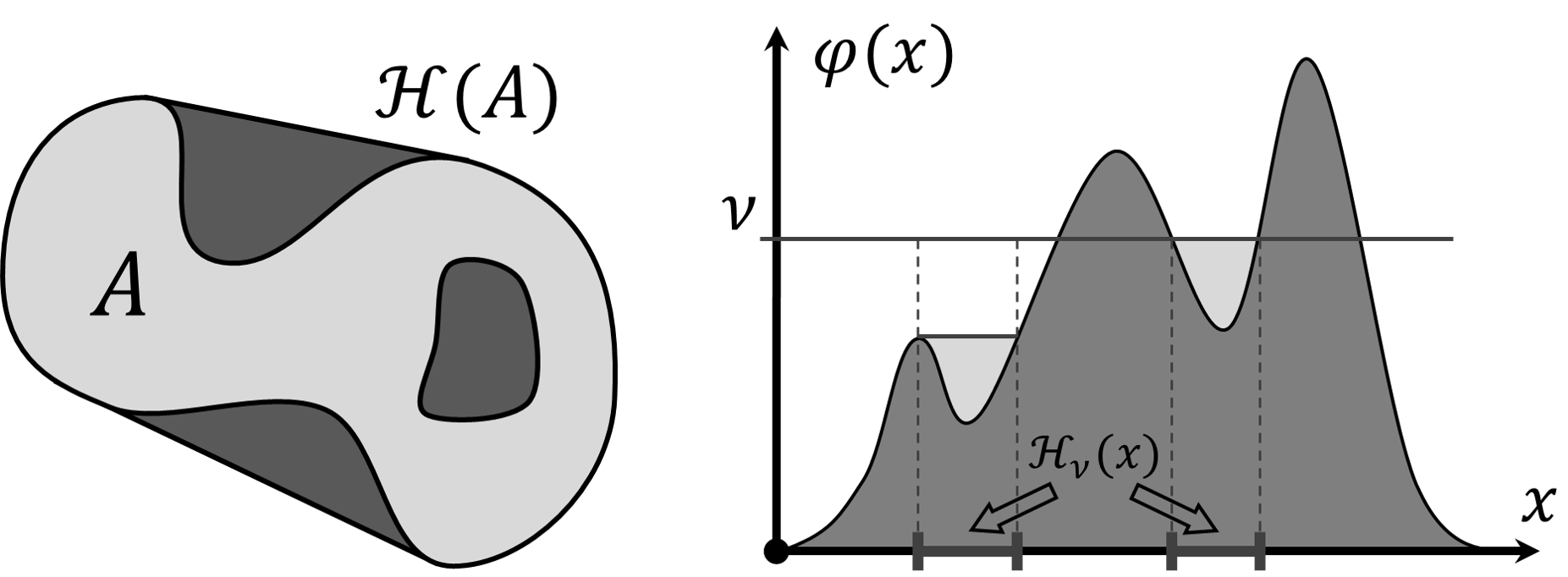}
\caption{\label{hollows} Illustration of the hollows.\newline Left: the hollows of a set $A$. Right: the hollows at level $\nu$ of a function $\varphi$.}
\end{figure} 

\begin{Defi} \indent\label{define_conv_ess}
Let $A$ be a measurable subset of $\RR_+$ (for the Lebesgue measure). We define the \emph{essential convex hull} of the set $A$ by
$$\mathrm{conv\;ess\;}A:=\bigcap_{\meas(A\triangle B)=0}\conv B,$$
where $\triangle$ is the symmetrical difference, $A\triangle B := (A\cup B)\setminus (A\cap B).$
\end{Defi}
This definition allows us to define properly the notion of \emph{hollow} that we adumbrated just before.
\begin{Defi} \indent\label{define_hollows}
Let $A$ be a measurable subset of $\RR$. The \emph{hollows} of the set $A$ is the set defined by 
$$\cH(A):=\bigcup_{\meas(A\triangle B)=0}\left(\conv B\cap\mathrm{conv\;ess\;}A\right)\setminus B.$$
\end{Defi}
This definition naturally extends to functions via the super-level sets (see Figure \ref{hollows}, Right).
\begin{Defi} \indent\label{define_L^p_hollows}
Let $\varphi\in\mfM_+(\RR_+)$. The \emph{hollows} of $\varphi$ at level $\nu$ are the set 
$$\cH_\nu(\varphi):=\bigcup_{\lambda<\nu}\cH\left(\{\varphi\geq\lambda\}\right).$$
\end{Defi}

\subsubsection{A definition of the tamping}
Here after is an equivalent description of the rearrangement by tamping. The main interest of this proposition is that the result above can be used to define the tamping for any function in $\mfM_+(\RR_+)$.
\begin{Prop} \indent\label{L^p_tamping}
Let $\varphi\in\widetilde{\cE_n}(\RR)$. Let us define for all $\nu>0$ \begin{equation}\label{define_x_nu}x_\nu(\varphi):=\mathrm{inf\;ess}\;\{\varphi\geq\nu\}\end{equation} and \begin{equation}\label{define_y_nu}y_\nu(\varphi):=x_\nu(\varphi)-\meas\Big(\cH_\infty(\varphi)\cap[0,x_\nu(\varphi)]\Big).\end{equation}
With these definitions we can describe the superlevel sets of the function $\varphi^\natural$ as
$$\{\varphi^\natural\geq\nu\}=\Big[\;y_\nu(\varphi),\;y_\nu(\varphi)+\meas\{\varphi\geq\nu\}\;\Big].$$
In other words,
$$y_\nu(\varphi)=x_\nu(\varphi^\natural).$$
\end{Prop}
\begin{proof} \indent
It is enough to consider the case $\varphi\in\cE_n(\RR)$.
The fact that these superlevel sets are segments whose length is equal to $\meas\{\varphi\geq\nu\}$ follows from the construction of the rearrangement by tamping. 

The value of the infimum of these segments is determined using the algorithm. We can see that the notion of hollows for a function in $\cE_n$ exactly coincides with the number of times a given cube is slid one step on the left during the tamping algorithm. More precisely, such a cube in place $(i,j)$ is slid on a distance which is the integer part of the quantity $\cH_j(\varphi)\cap[0,i]$. Indeed this quantity is the size of the hollows at level $j$ that we fill by moving cubes and a cube whose abscissa index is $i$ is not concerned by the filling of the hollows that are positioned at indices bigger than $i$. Similarly, what happens at layers bigger than $j$ does not interfere on what happens for the cube at $(i,j)$ but only the levels below. 

The conclusion of the demonstration then comes from the fact that, regarding the definition of $x_\nu(\varphi)$, we have
$$  \cH_\nu(\varphi)\cap[0,x_\nu(\varphi)]=\cH_\infty(\varphi)\cap[0,x_\nu(\varphi)].$$
\end{proof}
In the proposition above, the description of the super-level sets that we obtain remains well-defined for any function in $\mfM_+(\RR_+)$. From now, we can use the above proposition as the \emph{definition of the tamping}.
\begin{Defi} \label{definition_tamping_2}
Let $\varphi\in\mfM_+(\RR_+)$. We define the \emph{tamping} of $\varphi$, noted $\varphi^\natural$, as being the function of $\mfM_+(\RR_+)$ which super-level sets are 
$$\{\varphi^\natural\geq\nu\}=\Big[\;y_\nu(\varphi),\;y_\nu(\varphi)+\meas\{\varphi\geq\nu\}\;\Big],$$
where $y_\nu$ is defined by \eqref{define_x_nu} and \eqref{define_y_nu}.
\end{Defi}
To ensure that this defines a function, we have to check that $\mu\geq\nu$ implies
\begin{equation}\label{the_inclusion}
\Big[\;y_\mu(\varphi),\;y_\mu(\varphi)+\meas\{\varphi\geq\mu\}\;\Big]\;\subseteq\;\Big[\;y_\nu(\varphi),\;y_\nu(\varphi)+\meas\{\varphi\geq\nu\}\;\Big].
\end{equation}
\begin{Prop} The inclusion \eqref{the_inclusion} is verified for any $\varphi\in\mfM_+(\RR_+)$. \end{Prop}
\begin{proof} Let $\mu\geq\nu$. On the one hand we have
\begin{equation}\label{inclusion1}
\begin{split}
y_\mu(\varphi)-y_\nu(\varphi)\;&=\;\meas\Big([x_\nu(\varphi),x_\mu(\varphi)]\Big)-\meas\Big(\cH_\infty(\varphi)\cap[x_\nu(\varphi),x_\mu(\varphi)]\Big)\\
&=\;\meas\Big(\cH_\infty(\varphi)^c\cap[x_\nu(\varphi),x_\mu(\varphi)]\Big)\;\geq0.
\end{split}
\end{equation}
On the other hand,
\begin{equation}\label{inclusion2}
\begin{split}
\meas\Big(\cH_\infty(\varphi)^c\cap[x_\nu(\varphi),x_\mu(\varphi)]\Big)\;\leq\;\meas\Big(\{\nu<\varphi\}\cap[x_\nu(\varphi),x_\mu(\varphi)]\Big)\\
=\meas\Big(\{\nu<\varphi\leq\mu\}\cap[x_\nu(\varphi),x_\mu(\varphi)]\Big)\;\leq\;\meas\Big(\{\nu<\varphi\leq\mu\}\Big).
\end{split}
\end{equation}
Combining \eqref{inclusion1} and \eqref{inclusion2} we get
\begin{equation}\label{inclusion3}
y_\mu(\varphi)+\meas\{\varphi\geq\mu\}\;\leq\;y_\nu(\varphi)+\meas\{\varphi\geq\nu\}.
\end{equation}
Equations \eqref{inclusion1} and \eqref{inclusion3} give \eqref{the_inclusion}.
\end{proof}
Now that the tamping is well-defined, we can state the unimodality property, the Schwarz rearrangement inequality and the preservation of the Dirichlet boundary condition which are three of the four main properties that we required in Section \ref{sec:problem of tamping}.
\begin{Prop}[Unimodality]
Let $\varphi\in\mfM_+(\RR_+)$. Then there exist $s\in\RR_+$ such that $\varphi^\natural$ is non-decreasing on $[0,s]$ and non-increasing on $[s,+\infty)$.
\end{Prop}
\begin{proof}
This follows from the fact that the super-level sets of $\varphi^\natural$ are segments.
\end{proof}
\begin{Prop}[Schwarz rearrangement inequality for the tamping] Let $\varphi\in\mfM_+(\RR_+)$,
$$\forall\;x\in\RR_+,\qquad\meas\Big([0,x]\cap\{\varphi\geq\nu\}\Big)\;\leq\;\meas\Big([0,x]\cap\{\varphi^\natural\geq\nu\}\Big).$$
\end{Prop}
\begin{proof}
This is a consequence of the fact that this rearrangement verifies $\forall\;\nu>0,\; x_\nu(\varphi^\natural)\leq x_\nu(\varphi)$ and $\{\varphi^\natural\geq\nu\}$ is a segment.
\end{proof}
\begin{Prop}[Preservation of the Dirichlet boundary condition]\label{prop:dirichlet}
Let $\varphi\in\cC^0(\RR_+,\RR_+)$ be absolutely continuous. Then, $\varphi(0)=\varphi^\natural(0)$.
\end{Prop}
This fact is proved in section \ref{section_of_proof}.

\subsection{Best non-decreasing upper bound}
In the first section devoted to the presentation of the problem, we emphasized the links that we want between the rearrangement by tamping and the Schwarz non-increasing rearrangement: the Schwarz inequality (Property \ref{Schwarz_propr}) and the Pólya – Szegő inequality (Theorem \ref{Polya_Szego}). In this section we propose another equivalent way to define the tamping that directly involves the Schwarz non-increasing rearrangement (called here after the double Schwarz formula). Although we will not work with this second definition in the rest of this paper, it remains per se interesting because it tells more about the links between the tamping and the Schwarz rearrangement.
\begin{Defi} 
Let $\varphi\in\mfM_+(\RR_+)$. We define the \emph{best non-decreasing upper bound} of $\varphi$ as being
\begin{equation}\label{up_formula}\varphi^\dag(x)\;:=\;\supess\Big(\varphi.\mathbbm{1}_{[0,x]}\Big).\end{equation}
\end{Defi}
\begin{Prop} \label{up_property}
The function $\varphi^\dag$ is the only function of
\begin{equation}\label{upup_arrows}\left\{\psi\in\mfM_+(\RR_+)\;:\;\begin{array}{l}
\varphi\leq\psi\;\;\text{almost everywhere,}\\
\psi\;\;\text{is non-decreasing,}
\end{array}
 \right\}
\end{equation}
that is minimal for the standard comparison of functions. 
\end{Prop}
\noindent This proposition is proved in Section \ref{section_of_proof}. 

We define
\begin{equation}\label{def_a}s(\varphi)\;:=\;\lim\limits_{\nu\rightarrow(\supess\varphi)^-}\supess\{\varphi>\nu\}.\end{equation}
The value $s(\varphi)$ can be understood, roughly speaking, as being the supremum of the ``argmax'' of the function $\varphi$. It is not possible to define $s(\varphi)$ directly this way because $\varphi$ is only measurable and defined almost everywhere. We now define
\begin{equation}\label{def_alpha}
\sigma(\varphi)\;:=\;s(\varphi)-\meas\Big(\{\varphi\;\neq\;\varphi^\dag\}\cap[0,s(\varphi)]\Big).
\end{equation}
The purpose of this definition is the fact that we actually have $\sigma(\varphi)=s(\varphi^\natural)$ (this equality is given by the proof of the following lemma).

\begin{Lemm}[Double Schwarz formula]\label{double} 
Let $\varphi\in\mfM_+(\RR_+)$. Define $s(\varphi)$ and $\sigma(\varphi)$ respectively by~\eqref{def_a} and~\eqref{def_alpha}. Suppose that $s(\varphi)<+\infty$. Then, 
\begin{equation}\label{big_formula}
\varphi^\natural(x)\;=\;\left\{\begin{array}{ll}
\Big(\varphi.\mathbbm{1}_{\{\varphi=\varphi^\dag\}}\Big)^{\!\!\ast}\big(\sigma(\varphi)-x\big) &\;\mathrm{if}\quad x\leq \sigma(\varphi)\\\\
\Big(\varphi.\mathbbm{1}_{\{\varphi\neq\varphi^\dag\}}\Big)^{\!\!\ast}\big(x-\sigma(\varphi)\big)&\quad\mathrm{otherwise,}
\end{array} 
\right.
\end{equation}
where the superscript $\ast$ refers to the Schwarz non-increasing rearrangement.
\end{Lemm}

{\includegraphics[width=9cm]{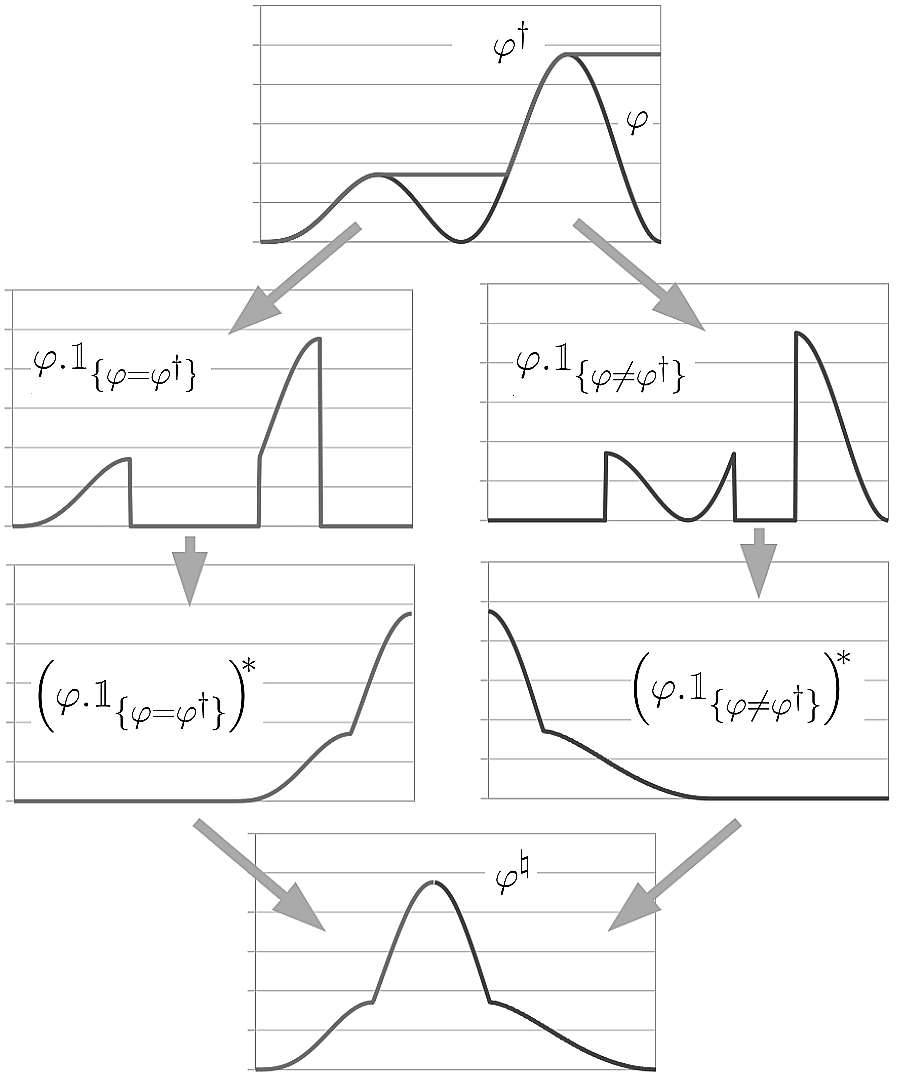}\centering\nopagebreak
\captionof{figure}{Tamping process and Schwarz symmetrization. Tamping of the function $x~\!\!\mapsto~\!\!x.\sin^2(2\pi x).\mathbbm{1}_{[0,1]}(x)$ computed using the double Schwarz formula \eqref{big_formula}.}
\label{uparrow}}~\newline\indent
The proof of formula \eqref{big_formula} is given in Section \ref{section_of_proof}. If we note $\widehat{\varphi}$ the expression on the right hand side of \eqref{big_formula}, the main steps of the proof are the following.
\begin{enumerate}
\item Establish that $\{\varphi\neq\;\varphi^\dag\}\cap[0,s(\varphi)]=\cH_\infty(\varphi)\cap[0,s(\varphi)].$
\item Conclude that $x_\nu(\varphi^\natural)=x_\nu(\widehat{\varphi}).$
\item Establish that $\meas\{\varphi^\natural\geq\nu\}=\meas\{\widehat{\varphi}\geq\nu\}.$
\item Conclude that $\{\varphi^\natural\geq\nu\}=\{\widehat{\varphi}\geq\nu\}$ and then $\varphi^\natural=\widehat{\varphi}$.
\end{enumerate}

One interest of this formula is the fact that the notion of \emph{best non-decreasing upper bound} also manages to seize the notion of \emph{hollows} that we introduced at the previous section. This aspect is hidden in the formula but is more visible in the proof. The main idea behind this formula is the first point of the proof where we precise what we can obtain about the hollows of $\varphi$ when we know $\varphi^\dag$. Yet, the main aspect of Formula \eqref{big_formula} that makes it interesting is the fact that the Schwarz non-increasing rearrangement appears in the computation of the tamping. With such a formula, we can reformulate and interpret the tamping as a double Schwarz non-increasing rearrangement. This idea is illustrated at Figure \ref{uparrow} where the tamping of the function $x\longmapsto x.\sin^2(2\pi x).\mathbbm{1}_{[0,1]}(x)$ is proceeded using the Double Schwarz formula \eqref{big_formula} (this is the process that we used to obtain the plots of Figure \ref{xsin2x} in the section ``Presentation of the problem'').

\section{Main results about the tamping}

\subsection{Functional analysis}
To deal with rearrangements, we provide in this section some tools of functional analysis for the manipulations of the super-level sets of functions in $L^p_+(\RR_+)$. The aim is to obtain results about the links between convergence of the super-level sets and convergence in $L^p$.

\subsubsection{Results on convergence of the super-level sets}

\begin{Lemm}\label{crucial_inequality} Let $\Omega$ be a domain of $\RR^d$ and let $p\in[1,+\infty)$. Let $\varphi,$ $\psi$ in $L^p_+(\Omega)$. We have
\begin{equation}\label{crucru} \big\|\varphi-\psi\big\|_{L^p}^p\;\leq\; p\int_0^\infty\nu^{p-1}\meas\Big(\{\varphi\geq\nu\}\triangle\{\psi\geq\nu\}\Big)\,\mathrm{d}\nu,\end{equation}
where $\triangle$ is the symmetrical difference between two sets: $A\triangle B:=(A\cup B)\setminus(A\cap B)$.
Equality holds if and only if $p=1$ or if the support of the two functions are disjointed.
\end{Lemm}
Note that the right-hand side of Inequality \eqref{crucru} is finite because we have
\begin{align*}
&p\int_0^\infty\nu^{p-1}\meas\Big(\{\varphi\geq\nu\}\triangle\{\psi\geq\nu\}\Big)\,\mathrm{d}\nu\\
&\leq\;p\int_0^\infty\nu^{p-1}\Big(\meas\{\varphi\geq\nu\}+\meas\{\psi\geq\nu\}\Big)\,\mathrm{d}\nu\\
&=\;\|\varphi\|_{L^p}^p+\|\psi\|_{L^p}^p.
\end{align*}
The proof of Inequality \eqref{crucru} is given in Section \ref{section_of_proof} and it relies on an idea from \cite{CianchiFerone}. 

\begin{Lemm}\label{Cv_equiv} 
Let $\Omega$ be a domain of $\RR^d$ and let $p\in[1,+\infty)$. Let $\varphi$ and $(\varphi_n)_{n\in\NN}$ in $L^p_+(\Omega)$. Then $(\varphi_n)$ converges towards $\varphi$ in $L^p$ if and only if
\begin{equation}
p\int_0^\infty\nu^{p-1}\meas\Big(\{\varphi\geq\nu\}\triangle\{\varphi_n\geq\nu\}\Big)\,\mathrm{d}\nu\;\xrightarrow[n\rightarrow\infty]{}\;0.\label{L^p_superconv}\end{equation}
\end{Lemm}
The fact that the convergence in $L^p$ follows from \eqref{L^p_superconv} is a consequence of Lemma \ref{crucial_inequality}. The converse is proved in Section 4.

\begin{Coro}\label{super-level_cv} 
Let $\Omega$ be a domain of $\RR^d$ and let $p\in[1,+\infty)$. Let $\varphi\in L^p_+(\Omega)$ and let $\varphi_n\longrightarrow\varphi$ in $L^p$. Then, up to an omitted extraction of this sequence, we have for almost every $\nu\in\RR_+$
\begin{equation}
\meas\Big(\{\varphi\geq\nu\}\triangle\{\varphi_n\geq\nu\}\Big)\;\xrightarrow[n\rightarrow\infty]{}\;0.
\end{equation}
\end{Coro}
\begin{proof} 
Let $a>0$. We have
\begin{align}\label{41}
&p\int_a^\infty\nu^{p-1}\meas\Big(\{\varphi\geq\nu\}\triangle\{\varphi_n\geq\nu\}\Big)\,\mathrm{d}\nu\;\\&\geq\;p\,a^{p-1}\int_a^\infty\meas\Big(\{\varphi\geq\nu\}\triangle\{\varphi_n\geq\nu\}\Big)\,\mathrm{d}\nu,\label{42}
\end{align}
and then $\nu\longmapsto\meas\Big(\{\varphi\geq\nu\}\triangle\{\varphi_n\geq\nu\}\Big)$ belongs to $L^1(a,+\infty)$. But since \eqref{41} vanishes when $n\rightarrow\infty$ by Lemma \ref {Cv_equiv} then so does \eqref{42}. We conclude using a standard result \cite{Brezis} which states that the convergence in $L^1$ implies the convergence almost everywhere up to an extraction. Considering a sequence $a_k\longrightarrow0^+$ and successive extractions completes the proof.\end{proof}
\subsubsection{Results on convergence in $L^p$}

\begin{Lemm}[Weak convergence in $L^p$]\label{before_crucial_lemma}
Let $\Omega$ be a domain of $\RR^d$ and let $p\in(1,+\infty)$.
Let $\varphi\in L^p_+(\Omega)$ and let $\varphi_n\in L^p_+(\Omega)$ such that
$\|\varphi_n\|_{L^p}$ is a bounded sequence and
$$\meas\Big(\{\varphi_n\geq\nu\}\triangle\{\varphi\geq\nu\}\Big)\;\xrightarrow[n\rightarrow+\infty]{}\;0,$$
for almost every $\nu>0$.
Then $(\varphi_n)$ weakly converges towards $\varphi$ in $L^p$.
\end{Lemm}
The proof of this lemma is done in Section \ref{section_of_proof}. The two main ingredients of the proof are the inequality given at Lemma \ref{crucial_inequality} and the Lebesgue dominated convergence theorem \cite{LiebLoss}. For the case $p=1$, a counter-example is given by the sequence $\varphi_n:=\frac{1}{n}\mathbbm{1}_{[0,n]}\in L^1$ which satisfies all the hypothesis but which is not weakly converging in $L^1$ (integrate $\varphi_n$ against a constant function for instance). Nevertheless, the case $p=1$ turns out to be true for the strong version of this lemma.

\begin{Lemm}[Strong convergence in $L^p$]\label{crucial_lemma} Let $\Omega$ be a domain of $\RR^d$ and let $p\in[1,+\infty[$. Let $\varphi\in L^p_+(\Omega)$ and let $\varphi_n\in L^p_+(\Omega)$ such that
\begin{equation}\label{one hypothesis}\|\varphi_n\|_{L^p}\longrightarrow\|\varphi\|_{L^p}\end{equation} and
\begin{equation}\label{another hypothesis} \meas\Big(\{\varphi_n\geq\nu\}\triangle\{\varphi\geq\nu\}\Big)\;\xrightarrow[n\rightarrow+\infty]{}\;0,\end{equation}
for almost every $\nu>0$. Then,
$$\big\|\varphi_n-\varphi\big\|_{L^p}\;\xrightarrow[n\rightarrow+\infty]{}\;0.$$
\end{Lemm}

By the Radon-Riesz lemma \cite{Brezis}, a sequence $(\varphi_n)$ weakly converging towards $\varphi$ in $L^p$ for $p\in(1,+\infty)$ is strongly converging if and only if the $L^p$ norm of $\varphi_n$ converges towards the $L^p$ norm of $\varphi$. The proof in the case $p\neq1$ therefore follows from Lemma \ref{before_crucial_lemma} with the Radon-Riesz lemma. The case $p=1$ is proved in Section \ref{section_of_proof}. Note that Hypothesis \eqref{another hypothesis} directly involves the super-level sets of the manipulated functions, while Hypothesis \eqref{one hypothesis} is easy to verify using the layer-cake representation \eqref{layercake_Lp}. This makes Lemma 3.5 be a good tool in the context of rearrangements.

This lemma is false in the case $p=+\infty$. Consider for instance the sequence  $\varphi_n:=\mathbbm{1}_{[0,1+1/n]}$ as a counter-example.

\subsection{Topological results in $L^p_+(\RR_+)$.}

\subsubsection{Compactness result for the tamping in $L^p_+(\RR_+)$}
Before starting our result of convergence for the tamping, we provide some compactness results for the tamping in $L^p_+(\RR_+)$.

\begin{Lemm}[Compactness of the super-level sets]\label{compactness_result} Let $p\in[1,+\infty)$ and let $(\varphi_n)\in L^p_+(\RR_+)$ be a bounded sequence in $L^p$ such that
\begin{equation}
\exists\;\mu>0,\qquad\limsup\limits_{n\to+\infty}\;x_\mu(\varphi_n)\;<\;+\infty.
\end{equation}
Then there exists a function $\psi$ such that, up to an omitted extraction and for almost every $\nu>0$,
\begin{equation}
\meas\Big(\{\varphi_n^\natural\geq\nu\}\triangle\{\psi\geq\nu\}\Big)\;\xrightarrow[n\to+\infty]{}\;0.
\end{equation}
\end{Lemm}
The proof of this lemma is provided in Section \ref{section_of_proof}. Now, combining Lemma \ref{before_crucial_lemma} and Lemma \ref{compactness_result} we get the following corollary.
\begin{Coro}[Weak compactness for the tamping]\label{compactness_weak}
Let $p\in(1,+\infty)$ and let $(\varphi_n)\in L^p_+(\RR_+)$ be a bounded sequence such that
\begin{equation}\label{the_compactness_hypothesis}
\exists\;\mu>0,\qquad\limsup\limits_{n\to+\infty}\;x_\nu(\varphi_n)\;<\;+\infty.
\end{equation}
Then the sequence $(\varphi_n^\natural)$ is weakly compact in $L^p$.
\end{Coro}
Hypothesis \eqref{the_compactness_hypothesis} may look quite technical and little intuitive. Nevertheless, in the case of a sequence of functions $(\varphi_n)$ whose support remains contained in a given compact $K$, this hypothesis is automatically verified if and only if $(\varphi_n)$ does not converges to the null function. Since the case of the convergence towards the null function is easy to handle with, we can conclude that the following corollary holds.
\begin{Coro}[Weak compactness for the tamping on compact support]\label{compactness_weak_compact}
Let $p\in[1,+\infty)$ and let $(\varphi_n)\in L^p_+(\RR_+)$ be a bounded sequence in $L^p$. We suppose that there exists $K\subseteq\RR_+$ a compact such that
\begin{equation}
\forall n\in\NN,\quad\supp\varphi_n\;\subseteq\;K.
\end{equation}
Then the sequence $(\varphi_n^\natural)$ is weakly compact in $L^p$.
\end{Coro}
If we now go back to Lemma~\ref{compactness_result} and combine it with  Lemma \ref{crucial_lemma} and the fact that $\|\varphi_n\|_{L^p}\!\!~=~\!\|\varphi_n^\natural\|_{L^p}$, we get the following result for strong compactness.
\begin{Coro}[Strong compactness for the tamping]\label{compactness_strong} Let $p\in[1,+\infty)$. Let $(\varphi_n)\in L^p_+(\RR_+)$ such that $\varphi_n$ converges in $L^p$ towards a function $\varphi$.
Then, up to an extraction, the sequence $\varphi_n^\natural$ converges strongly in $L^p$ towards a function $\psi$ that verifies 
\begin{equation}
\meas\{\psi\geq\nu\}\;=\;\meas\{\varphi\geq\nu\}.
\end{equation}
\end{Coro}

Without more hypothesis on the sequence $(\varphi_n)$, we cannot obtain convergence of the sequence $(\varphi_n^\natural)$ towards the function $\varphi^\natural$. We have for instance the counter-example (see Figure \ref{contre_exemple_1}),
$$\left\{\begin{array}{l}\varphi(x):=\mathbbm{1}_{[1,2]}(x),\\\\
\varphi_n(x):=\mathbbm{1}_{[1,2]}(x)+\mathbbm{1}_{\left[0,\,\frac{1}{n}\right]}(x).
\end{array}\right. $$
\begin{figure}[h]\centering
\includegraphics[width=11cm]{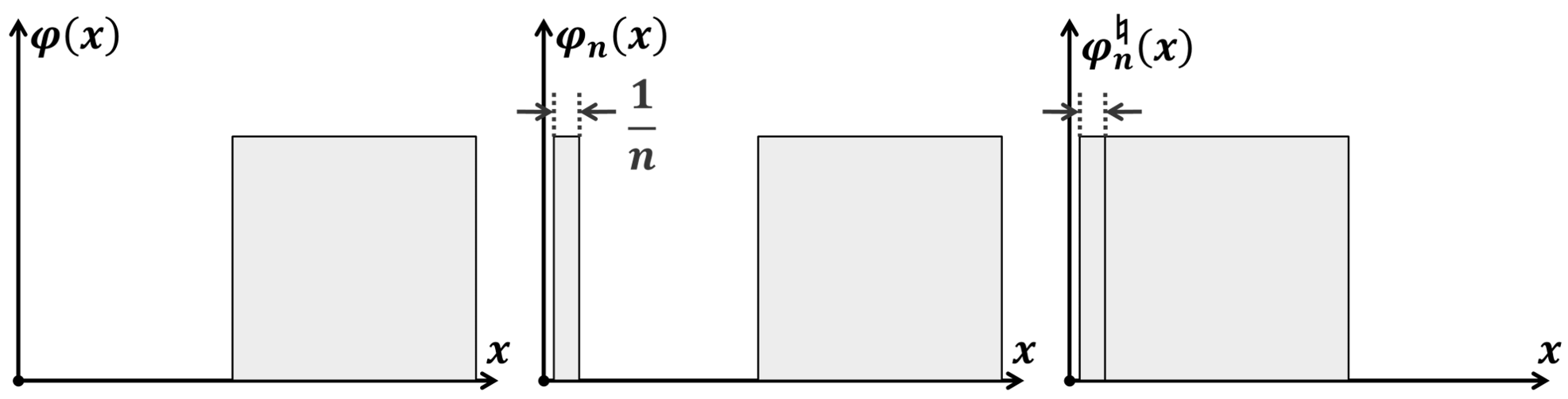}
\caption{\label{contre_exemple_1} Illustration of the first counter-example for the convergence.}
\end{figure}\newline
The sequence $(\varphi_n)$ converges towards $\varphi$ in $L^p$ whereas this is not the case for $(\varphi_n^\natural)$ which converges towards a different function in $L^p_+(\RR_+)$.

Another interesting counter-example is the following (see Figure \ref{contre_exemple_2}), 
$$\left\{\begin{array}{l}\varphi(x):=\mathbbm{1}_{[0,2]}(x)+2.\mathbbm{1}_{[2,3]}(x),\\\\
\varphi_n(x):=\mathbbm{1}_{[0,1]}(x)+\left(1-\frac{1}{n}\right)\mathbbm{1}_{[1,2]}(x)+2.\mathbbm{1}_{[2,3]}(x).
\end{array}\right. $$
\begin{figure}[h]\centering
\includegraphics[width=11cm]{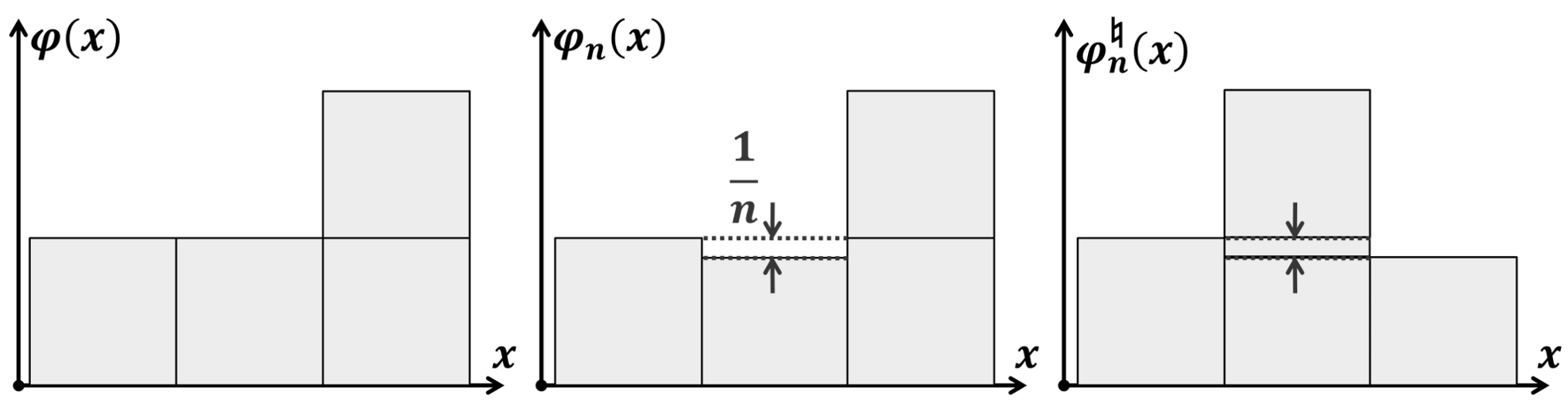}
\caption{\label{contre_exemple_2} Illustration of the second counter-example for the convergence.}
\end{figure}\newline
Once again we choose a function $\varphi$ that is already tamped and we exhibit a sequence $\varphi_n$ that is converging towards the function $\varphi$ but such that the sequence of tamped functions $\varphi_n^\natural$ is not converging towards the tamped function $\varphi^\natural$ ($=\varphi$ here). 

The main common point between these two counter-examples that may seem quite different at first sight is that in both cases the approximating functions are in such a way that, roughly speaking, ``\emph{there is a new hollow that appears in front of the big bump}'' and this hollow never vanishes as $n$ grows. Although rough this observation may seem, this is the key idea to obtain a sufficient condition for convergence towards $\varphi^\natural$.

\subsubsection{A convergence result for the tamping}
\begin{Theo}[Convergence result] \label{continuity_theorem_1D} Let $p\in[1,+\infty)$ and let $(\varphi_n)_{n\in\NN}$ in  $L^p_+(\RR_+)$ converging towards $\varphi$. We suppose that we have the \emph{local hollows convergence condition}, that is
\begin{equation}\label{44}
\forall a>0,\qquad\meas\Big((\cH_\infty(\varphi_n)\triangle\cH_\infty(\varphi))\cap[0,a]\Big)\;\xrightarrow[n\rightarrow+\infty]{}\;0.
\end{equation}
Then, up to an extraction we have
\begin{equation}\label{45}\|\varphi_n^\natural-\varphi^\natural\|_{L^p}\;\xrightarrow[n\rightarrow+\infty]{}\;0.\end{equation}
Moreover, if for almost every $\nu\geq0$
\begin{equation}
\meas\Big(\{\varphi_n\geq\nu\}\triangle\{\varphi\geq\nu\}\Big)\;\xrightarrow[n\rightarrow+\infty]{}\;0,\label{46}
\end{equation}
then it it is also the case for $(\varphi_n^\natural)$ and the full sequence converges in $L^p$ towards $\varphi^\natural$.
\end{Theo}
Since it is possible to obtain Condition \eqref{46} from Corollary \ref{super-level_cv}, it is enough to prove the last assertion. The proof is provided in Section 4. This condition \eqref{44} is actually not very surprising if we consider that the tamping is defined through the hollows of the function. If we want to converge towards the tamping of $\varphi$ then it is natural to impose that the hollows are matching asymptotically. 

We can now use this convergence theorem to obtain a convergence result on voxel functions.
\begin{Coro}\label{continuity_result_1D_voxels} Let $p\in[1,+\infty)$ and let $\varphi\in L^p_+(\RR_+)$. There exist $\varphi_n\in\widetilde{\cE_n}(\RR_+)$ such that
$$\big\|\varphi_n-\varphi\big\|_{L^p}\xrightarrow[n\rightarrow\infty]{}0,\qquad\quad
\text{and} \quad\qquad\big\|\varphi_n^\natural-\varphi^\natural\big\|_{L^p}\xrightarrow[n\rightarrow\infty]{}0.$$
This sequence also verifies the convergence of the superlevel sets
\begin{align}
\meas\Big(\{\varphi_n\geq\nu\}\triangle\{\varphi\geq\nu\}\Big)\;\xrightarrow[n\rightarrow+\infty]{}\;0,\\
\meas\Big(\{\varphi_n^\natural\geq\nu\}\triangle\{\varphi^\natural\geq\nu\}\Big)\;\xrightarrow[n\rightarrow+\infty]{}\;0,
\end{align}
for almost every $\nu>0$.
\end{Coro}
\begin{proof} 
We define the voxel function $\varphi_n\in\widetilde{\cE_n}(\RR)$, following the construction (and the notations) of Section \ref{voxel_things}.
We consider cubes $a(i,j)$ of vanishing size $\lambda_n\!\times\mu_n$. Since by definition of $\widetilde{\cE_n}(\RR)$, the number of cubes is at most $n\!\times\! n$, to obtain the convergence in $L^p$ we have to impose the following decay rates conditions,
\begin{equation}
n.\lambda_n\xrightarrow[n\rightarrow\infty]{}\;+\infty\qquad\mathrm{and}\qquad
n.\mu_n\xrightarrow[n\rightarrow\infty]{}\;+\infty.
\end{equation}
Then we define the boolean function $\Gamma$ by
\begin{equation}
\Gamma(i,j)\;:=\;\left\{\begin{array}{cl}
0&\mathrm{if\;}\meas(a(i,j)\cap\mathrm{hypo}(\varphi))=0,\\
1&\mathrm{otherwise},
\end{array} 
\right.\end{equation}
where the \emph{hypograph} of a function $\varphi \in \mfM_+(\RR_+)$ is given by
$$
\mathrm{hypo}(\varphi):=\big\{(x,y)\in\RR_+\times\RR:\varphi(x)\leq y\}.
$$
 
For every $a>0$, by construction we have that the sequence $\{\varphi_n\geq\nu\}\cap[0,\,a]$ is non-increasing (for the inclusion) and $\cH(\{\varphi_n\geq\nu\})\cap[0,\,a]$ is non-decreasing when $n$ is large enough (depending on $\nu$ and $a$). We also have $x_\nu(\varphi_n)$ converging towards $x_\nu(\varphi)$. The hypothesis of Theorem \ref{continuity_theorem_1D}, including \eqref{46}, are then verified and the conclusion follows.
\end{proof}

\subsection{Pólya–Szegő inequality for the tamping}\label{subsection24}
We show in this section that it is possible to decrease the $L^p$ norm of the derivative, also called the $W^{1,p}$ half-norm, using the rearrangement by tamping. In other words, we prove that the rearrangement by tamping satisfies a Pólya–Szegő inequality.
\subsubsection{Piece-wise linear approximation}
Since the tamping is constructed with functions in $\cE_n(\RR)$, whose derivatives are not in $L^p$, these half-norms are computed by approximation. For a function $\varphi_n\in\cE_n(\RR)$, we define its piece-wise linear and continuous approximation by: 
\begin{equation}\label{define_piece_wise_linear} (\Lambda\varphi_n)(x):=\sum_{i=0}^{n}\Big[\varphi(i+1).\Big(x-i\Big)+\varphi(i).\Big((i+1)-x\Big)\Big]\mathbbm{1}_{\left[i,i+1\right]}(x).
\end{equation}
The operator $\Lambda$ is a bijection between $\cE_n(\RR)$ and $\Lambda\cE_n(\RR)$. The operator $\Lambda$ extends to $\widetilde{\cE_n}(\RR)$ with a natural definition. The functions $\varphi_n\in\cE_n(\RR)$ and $\Lambda\varphi_n$ coincide on the set $\displaystyle\ZZ$. When $n$ goes to $\infty$, the sequences $(\varphi_n)$ and $(\Lambda\varphi_n)$ are such that if one converges then the other converges towards the same limit in $L^p$. Moreover, if we choose $(\varphi_n)$ approximating a given function $\varphi\in W^{1,p}$, we can expect, using standard theory on piece-wise linear continuous approximation \cite{CianchiFerone}, that the sequence $(\Lambda\varphi_n)$ is converging $W^{1,p}$ towards $\varphi$.

\subsubsection{Pólya–Szegő inequality for the tamping}
We use the notation $\nabla$ for the derivative of the functions, although we work in dimension 1, in order to make the formulas easier to read.
\begin{Theo}\label{decrease_W_1_p} \indent
Let $\varphi\in W^{1,p}_+(\RR_+)$. We have
\begin{equation}\label{decrease_W_1_p_1}\int_{\cH_\infty(\varphi)}|\nabla\varphi|^p\;\leq\;\int_{\RR_+}|\nabla\varphi|^p-|\nabla\varphi^\natural|^p.\end{equation}
Or equivalently, 
\begin{equation}\label{decrease_W_1_p_2}\int_{\RR_+}|\nabla\varphi^\natural|^p\;\leq\;\int_{\RR_+\setminus\cH_\infty(\varphi)}|\nabla\varphi|^p.\end{equation}
\end{Theo}
This theorem emphasizes the fact that we improve the $L^p$ norm of the derivative by ``\emph{filling the hollows}'' because the error term only involves the derivative of $\varphi$ inside the hollows of $\varphi$. 
The proof of this theorem is provided in Section \ref{section_of_proof}. It relies on the Riemannian point of view on the tamping on voxel functions $\varphi_n$ that we defined in the proof of Lemma \ref{defines_a_new_function}. Using this point of view, we can first estimate how does vary the $L^p$ norm of the derivative of $\Lambda\varphi_n$ during one step of the elementary tamping algorithm. Then, we iterate the estimate given by one step of the tamping algorithm to obtain, for a well-chosen sequence $(\varphi_n)\in\widetilde{\cE_n}(\RR)$, the inequality
\begin{equation}\int_{\cH_\infty(\varphi_n)}|\nabla\Lambda\varphi_n|^p\;\leq\;\int_{\RR_+}|\nabla\Lambda\varphi_n|^p-|\nabla\Lambda\varphi_n^\natural|^p\;+\;\underset{n\rightarrow\infty}{o}\!\!(1),\end{equation}
and we can conclude by passing to the limit $n\rightarrow\infty$.

\begin{Coro}[Pólya–Szegő inequality for the tamping] \indent
Let $\varphi\in W^{1,p}_+(\RR_+)$. We have
\begin{equation}
\int_{\RR_+}|\nabla\varphi^\natural|^p\;\leq\;\int_{\RR_+}|\nabla\varphi|^p,
\end{equation}
with equality if and only if $\varphi=\varphi^\natural$ almost everywhere.
\end{Coro}

This corollary is obtained directly from Theorem \ref{decrease_W_1_p}. The estimate \eqref{decrease_W_1_p_2} also gives the equality case above using the fact that $\varphi=\varphi^\natural$ is equivalent to $\cH_\infty(\varphi)=\emptyset$ and the fact that $|\nabla\varphi|$ cannot worth identically $0$ inside the hollows (in dimension $1$ a $W^{1,p}$ function is Hölder continuous).

\subsection{About a Riesz inequality for the tamping}\label{section4}
Now that we have a Pólya–Szegő inequality for the tamping, one natural thing to expect is that we get a Riesz rearrangement inequality for the tamping similar to Theorem \ref{Riesz} because we can expect that the tamping also decreases the $W^{s,p}$ half-norms. Indeed, as we evoked in Section \ref{section1}, this inequality gives that a rearrangement does not increase the $H^s$ norms. We explain why this inequality is actually false in the case of the rearrangement by tamping and we discuss in this section some aspects about the tamping and $H^s$ norms.

\subsubsection{Riesz rearrangement inequality: A counter-example}

Theorem \ref{Riesz} is false in the case of the tamping and we provide hereafter a counter-example. Such a result is not really surprising because we already know that the tamping does not verify the Hardy-Littlewood inequality (the cases presented at Figures \ref{contre_exemple_1} and \ref{contre_exemple_2} work as counter-examples since this inequality implies the continuity in $L^2$). As we explained at Subsection \ref{rearrangement inequalities}, the Riesz rearrangement inequality can be interpreted as a generalization of the Hardy-Littlewood inequality.

Here is a counter-example for the Riesz rearrangement inequality in the case of the tamping. Let $0<a<b<c<d<e$ and let $t\geq0$. If we make the supposition that $d\leq t\leq e-b$, then we have
\begin{align}
\int_{\RR_+}&\int_{\RR_+}\mathbbm{1}_{[0,e]}(x)\;.\;\Big(\mathbbm{1}_{[a,b]}+\mathbbm{1}_{[c,d]}\Big)(y)\;.\;\mathbbm{1}_{[-t,t]}(x-y)\;\mathrm{d}x\;\mathrm{d}y\\
&>\;\int_{\RR_+}\int_{\RR_+}\mathbbm{1}_{[0,e]}(x)\;.\;\mathbbm{1}_{[a,\,b-c+d]}(y)\;.\;\mathbbm{1}_{[-t,t]}(x-y)\;\mathrm{d}x\;\mathrm{d}y.
\end{align}
This is a counter-example for the Riesz rearrangement inequality for the tamping because we have 
\begin{equation}\label{counter_example}\Big(\mathbbm{1}_{[a,b]}+\mathbbm{1}_{[c,d]}\Big)^\natural=\mathbbm{1}_{[a,\,b-c+d]}.\end{equation}
The computation of the above inequality is provided in Section \ref{section_of_proof}.

\subsubsection{Decreasing the $H^s$ half-norm: A counter-example}
Working again on this counter-example, we can prove that the tamping sometimes fails to decrease the $H^s$ half-norms defined by
\begin{equation}
|\varphi|_{H^s}^2\;:=\;\int_\RR\int_\RR\frac{|\varphi(x)-\varphi(y)|^2}{|x-y|^{1+2s}}\,\mathrm{d}x\,\mathrm{d}y.
\end{equation}
If we define $\psi:=\mathbbm{1}_{[a,b]}+\mathbbm{1}_{[c,d]}+\mathbbm{1}_{[0,e]},$ then a direct computation gives that, for $s\in]0,\frac{1}{2}[$,
\begin{align*}
\|\psi\|^2_{H^s}\;&=\;\frac{1}{s\,\Big(\frac{1}{2}-s\Big)}\bigg[b^{1-2s}-a^{1-2s}+d^{1-2s}-c^{1-2s}+e^{1-2s}+(b-a)^{1-2s}\\&-(c-a)^{1-2s}+(c-b)^{1-2s}+(d-a)^{1-2s}
-(d-b)^{1-2s}+(d-c)^{1-2s}\\&+(e-a)^{1-2s}-(e-b)^{1-2s}+(e-c)^{1-2s}-(e-d)^{1-2s}\bigg],
\end{align*}
whereas
\begin{align*}
\|\psi^\natural\|^2_{H^s}\;=\;\frac{1}{s\,\Big(\frac{1}{2}-s\Big)}&\bigg[(b+d-c)^{1-2s}-a^{1-2s}+e^{1-2s}+(b+d-c-a)^{1-2s}\\&+(e-a)^{1-2s}-(e+c-d-b)^{1-2s}\bigg].
\end{align*}
Using a computer, if we take $s=\frac{1}{4}$ and $a=1$, $b=2$, $c=17$, $d=32$, $e=52$, we obtain that
$$\|\psi\|^2_{H^s}\;\approx\; 124.07\qquad\mathrm{and}\qquad\|\psi^\natural\|^2_{H^s}\;\approx\; 124.48,$$
which means that the rearrangement by tamping does not decrease the $H^\frac{1}{4}$ half norm in this case. We also find counter-examples for $s=0.2,$ $0.3,$  $0.35,$ $0.4,$ and $0.45$. 
This invites us to think that the tamping fails to decrease the $H^s$ half-norms for all $s<\frac{1}{2}$ (but we have no systematic counter-example yet). Nevertheless, whether the tamping decreases the $H^s$ half-norms or not when $s=\frac{1}{2}$ (or higher) remains unclear. The links between the problems of decreasing the $H^\frac{1}{2}$ half-norms in dimension $1$ and decreasing the $H^1$ half-norms in dimension $2$ are well known \cite{CaffarelliSilvestre} and then such a result on the tamping for the $H^\frac{1}{2}$ half-norms would be an important step on the question whether it is possible to extend this work to dimension $2$ or not.

\section{Proofs of the main results}\label{section_of_proof}
\subsection{Proof of Property \ref{prop:dirichlet}}
Let $\varphi\in\cC^0(\RR_+,\RR_+)$ be absolutely continuous. First, since $\forall\;\nu\geq 0,~x_\nu(\varphi^\natural)\leq x_\nu(\varphi)$, we deduce that
\begin{equation*}
\{\nu\in\RR_+/x_\nu(\varphi^\natural)=0\}\subseteq\{\nu\in\RR_+/x_\nu(\varphi)=0\}.
\end{equation*}
Thus, $\varphi^\natural(0)\geq\varphi(0).$ Therefore, by definition of the tamping (Definition \ref{definition_tamping_2}), we can conclude the proof if we obtain that
\begin{equation}
\label{viconte de valmont}
\forall\;\nu>\varphi(0),\qquad\meas\Big(\cH_\nu(\varphi)\cap[0,x_\nu(\varphi)]\Big)<x_\nu(\varphi).
\end{equation}
Indeed, in this case we have $x_\nu(\varphi^\natural)=x_\nu(\varphi)-\meas\Big(\cH_\nu(\varphi)\cap[0,x_\nu(\varphi)]\Big)>0$ and then $\varphi^\natural(0)\leq\varphi(0).$ Let $\nu>\varphi(0)$. We first recall the Luzin property which states that the image of a set of Lebesgue measure zero by an absolutely continuous function is of Lebesgue measure zero. Here $\varphi$ is absolutely continuous by hypothesis and 
\begin{equation*}
\varphi\bigg(\bigcup_{\mu\in[\varphi(0),\nu]}x_\mu(\varphi)\bigg)=[\varphi(0),\nu].
\end{equation*}
We infer that $\bigcup_{\mu\in[\varphi(0),\nu]}x_\mu(\varphi)$ has a positive measure. On the other hand, by definition of $\cH_\lambda(\varphi)$ (Definition \ref{define_L^p_hollows}) we have,
\begin{equation*}
\forall\;\mu<\nu,\quad x_\mu(\varphi)\notin\cH_\nu(\varphi).
\end{equation*}
Finally, we observe that $\mu\mapsto x_\mu(\varphi)$ is nondecreasing. All these facts together give
\begin{equation}
\meas\Big(\cH_\nu(\varphi)\cap[0,x_\nu(\varphi)]\Big)\leq\meas([0,x_\nu(\varphi)])-\meas\bigg(\bigcup_{\mu\in[\varphi(0),\nu]}x_\mu(\varphi)\bigg)<x_\nu(\varphi).
\end{equation}
and hence \eqref{viconte de valmont} is proved.\qed

\subsection{Proof of Property \ref{up_property}}
Let $\varphi^\dag(x):=\supess(\varphi\mathbbm{1}_{[0,x]})$. By definition we have
\begin{equation}
\varphi^\dag\in\;\left\{\psi\in\mfM_+(\RR_+)\;:\;\begin{array}{l}
\varphi\leq\psi\;\;\text{almost everywhere,}\\
\psi\;\;\text{is non-decreasing.}
\end{array}
 \right\}.\label{set64}
\end{equation}

We want to explain why $\varphi^\dag$ is a minimizer of this set (for the usual comparison of functions) and is the unique one. Considering the uniqueness, suppose that $\psi$ and $\chi$ verify the constraints. Then the function $x\mapsto\min\{\psi(x),\chi(x)\}$ also verifies the constraints and is strictly lower than at least one of the two unless $\psi=\chi$. For the existence, we have to explain why all the functions inside the considered set are higher than $\varphi^\dag$. Suppose that we have a function $\psi$ in the set \eqref{set64} such that on some bounded set $A$ of measure non zero \begin{equation}\label{supess_strict}\supess(\psi.\mathbbm{1}_{A})<\supess(\varphi.\mathbbm{1}_{A}).\end{equation} We note $a:=\supess A$.
\begin{itemize}
\item Case 1: Suppose that for every $\varepsilon>0$ holds
\begin{equation}\label{uparrow_case1}
\supess\Big(\varphi.\mathbbm{1}_{[0,a]}\Big)\;=\;\supess\Big(\varphi.\mathbbm{1}_{[a-\varepsilon,a]}\Big).\end{equation}
In this case, combining \eqref{supess_strict} and \eqref{uparrow_case1}, we obtain that there exists an $\varepsilon>0$ small enough so that
\begin{equation}\label{plouf}
\psi(x)\;<\;\varphi(x)\qquad\mathrm{for\;almost\;every\;}x\mathrm{\;in\;}A\cap[a-\varepsilon,a].
\end{equation}
Since $a=sup\;ess\; A$, the set $A\cap[a-\varepsilon,a]$ is not of measure $0$ and then equation \eqref{plouf} is in contradiction with the fact that $\psi\geq\varphi$.
\item Case 2: In the other case, there exists an $a_0<a$ such that,
\begin{equation}\label{ouaf}
\forall\;\varepsilon\in\left]0,\,\frac{a-a_0}{2}\right[,\quad\supess\Big(\varphi.\mathbbm{1}_{[0,a]}\Big)\;=\;\supess\Big(\varphi.\mathbbm{1}_{[a_0-\varepsilon,a_0+\varepsilon]}\Big).\end{equation}
We use the fact that $A\subseteq[0,a]$ and \eqref{ouaf} becomes
\begin{equation}
\supess\left(\varphi.\mathbbm{1}_{A}\right)\;\leq\;\supess\Big(\varphi.\mathbbm{1}_{[a_0-\varepsilon,a_0+\varepsilon]}\Big)
\end{equation}
We now inject \eqref{supess_strict} in the equation above and we use $A\cap[a-\varepsilon,a]\subseteq A.$ We obtain
\begin{equation} \label{we get}
\supess\Big(\psi.\mathbbm{1}_{A\cap[a_0+\varepsilon,\,a]}\Big)\;<\;\supess\Big(\varphi.\mathbbm{1}_{[a_0-\varepsilon,a_0+\varepsilon]}\Big).
\end{equation}
Finally, the fact that $\psi\geq\varphi$ with \eqref{we get} gives
\begin{equation}
\supess\Big(\psi.\mathbbm{1}_{A\cap[a_0+\varepsilon,\,a]}\Big)\;<\;\supess\Big(\psi.\mathbbm{1}_{[a_0-\varepsilon,a_0+\varepsilon]}\Big).
\end{equation}
Since $meas(A\cap[a_0+\varepsilon,\,a])>0$, the above inequality gives a contradiction with the fact that $\psi$ is non-decreasing.
\end{itemize}
We conclude that $\varphi^\dag$ is the minimal function. \qed

\subsection{Proof of Lemma \ref{double}}
We recall the formula \eqref{big_formula} that we now prove,
\begin{equation}\varphi^\natural:=\left\{\begin{array}{ll}
\Big(\varphi.\mathbbm{1}_{\{\varphi=\varphi^\dag\}}\Big)^{\!\!\ast}\big(\sigma(\varphi)-x\big) &\;\mathrm{if}\quad x\leq \sigma(\varphi)\\\\
\Big(\varphi.\mathbbm{1}_{\{\varphi\neq\varphi^\dag\}}\Big)^{\!\!\ast}\big(x-\sigma(\varphi)\big)&\quad\mathrm{otherwise.}
\end{array} 
\right.\end{equation}
Call $\widehat{\varphi}$ the function in the right-hand side of this equality.

~\newline\emph{\mathversion{bold}\textbf{Claim 1: Up to a set of measure $0$, the following inclusion holds.}}\mathversion{bold}
\begin{equation}
\cH_\infty(\varphi)\subseteq\{\varphi\neq\;\varphi^\dag\}.
\end{equation}\mathversion{normal}
By definition of the hollows of the function $\varphi$ (Definition \ref{define_L^p_hollows}), it is enough to prove that we have for all $\nu>0$ the inclusion 
$$\cH\Big(\!\{\varphi\geq\nu\}\Big)\;\subseteq\;\{\varphi\neq\;\varphi^\dag\}.$$
Let $X\subseteq\cH\big(\{\varphi\geq\nu\}\big)$ be a compact set. Since $X\subseteq\RR$ the quantities
$$\infess X\qquad\mathrm{and}\qquad\supess  X$$
are well-defined.
Since $X$ is bounded and included in the hollows of the set $\{\varphi\geq\nu\}$, by definition of the hollows of a set (Definition \ref{define_hollows}) there exist two sets $A$ and $B$ of positive measure and included in $\{\varphi\geq\nu\}$ such that
$$\supess A\;\leq\;\infess X\;\leq\;\supess  X\;\leq\;\infess B.$$
The definition of the hollows also gives
\begin{equation}\label{inclusion_holes_level_sets}
\cH\Big(\{\varphi\geq\nu\}\Big)\;\subseteq\;\{\varphi\geq\nu\}^c\;=\;\{\varphi<\nu\}.\end{equation}
Since $\varphi\geq\nu$ on $A$, the fact that $\varphi^\dag$ is non-decreasing combined with \eqref{inclusion_holes_level_sets} implies that $\varphi\neq\varphi^\dag$ on $X$. Since $X$ is any compact subset of $\cH\big(\{\varphi\geq\nu\}\big)$, Claim 1 is proved.

~\newline\emph{\textbf{Claim 2: Up to a set of measure zero, the following inclusion holds.}}
\mathversion{bold}$$\{\varphi\neq\;\varphi^\dag\}\cap[0,s(\varphi)]\subseteq\cH_\infty(\varphi).$$\mathversion{normal}
It is false in general that $\{\varphi\neq\;\varphi^\dag\}\subseteq\cH_\infty(\varphi)$ but this inclusion becomes true if we only consider what happens on $[0,s(\varphi)]$. Recall that $s(\varphi)$, defined by \eqref{def_a}, must be understood as being the smallest element in \emph{argmax} $\varphi$. We are going to prove that for any given $\delta>0$ and $\varepsilon>0$,
\begin{equation}\label{delta_inclusion}
D_{\delta,\varepsilon}:=\{\varphi\;\leq\;\varphi^\dag-\delta\}\cap[0,s(\varphi)-\varepsilon]\;\subseteq\;\cH_\infty(\varphi).
\end{equation}
Let $x$ in $D_{\delta,\varepsilon}$ and
\begin{equation}
E_{\delta,\varepsilon}:=[s(\varphi)-\varepsilon,+\infty[\;\cap\;\Big\{\varphi\;\geq\;\supess\;\varphi-\frac{\delta}{2}\Big\}.
\end{equation}
By definition of $s(\varphi)$, the measure of $E_{\delta,\varepsilon}$ is positive. If $b\in E_{\delta,\varepsilon}$ then $b\geq x$ and
\begin{align}
\varphi(x)&\leq\;\varphi^\dag(x)-\delta\;=\;\supess\varphi.\mathbbm{1}_{[0,x]}-\delta\;\leq\;\supess\varphi-\delta\\&\leq\;\varphi(b)-\frac{\delta}{2}.\label{major_varphi_right}
\end{align}
Moreover, there exists a set $A_{\delta}\subseteq[0,x]$ of measure non zero such that for all $a\in A_{\delta},$
\begin{equation}\label{A_delta}
\varphi(a)\;\geq\;\supess\Big(\varphi.\mathbbm{1}_{[0,x]}\Big)\;-\;\frac{\delta}{2}\;=\;\varphi^\dag(x)\;-\;\frac{\delta}{2}.
\end{equation}
Equation \eqref{A_delta} and the fact that $\varphi(x)\leq\;\varphi^\dag(x)-\delta$ give
\begin{equation}\label{major_varphi_left}
\varphi(x)\;\leq\;\varphi(a)-\frac{\delta}{2}.
\end{equation}
The upper bounds \eqref{major_varphi_right} and \eqref{major_varphi_left} with the definition of the hollows (Definition \ref{define_hollows}) give the inclusion \eqref{delta_inclusion}. Since $\delta>0$ and $\varepsilon>0$ are arbitrary, the claim is proved.

~\newline\emph{\textbf{\mathversion{bold}Claim 3: The essential infima of the super-level sets of $\varphi^\natural$ and $\widehat{\varphi}$ coincide. In other words, \begin{equation}  x_\nu(\varphi^\natural)=x_\nu(\widehat{\varphi}).
\end{equation}}}
First, $\widehat{\varphi}$ is non-decreasing on $[0,\sigma(\varphi)]$ and non-increasing on $[\sigma(\varphi),+\infty[$ and therefore for almost every $\nu>0$
\begin{equation}\label{x_nu_hat}
x_\nu(\widehat{\varphi})\;=\;x_\nu\bigg(x\in[0,\sigma(\varphi)]\longmapsto\Big(\varphi.\mathbbm{1}_{\{\varphi=\varphi^\dag\}}\Big)^{\!\!\ast}\big(\sigma(\varphi)-x\big)\bigg),
\end{equation}
where $x_\nu(\varphi):=\infess\{\varphi\geq\nu\}$. We now remark that the hypothesis $s(\varphi)<+\infty$ implies that all the quantities that we manipulate here after are finite and then the substractions are well-defined. Thus we can rewrite \eqref{x_nu_hat} as
\begin{equation}\label{x_nu_hat_hat}
x_\nu(\widehat{\varphi})\;=\;\sigma(\varphi)-\;\supess\bigg\{x\in\RR_+\;:\;\Big(\varphi.\mathbbm{1}_{\{\varphi=\varphi^\dag\}}\Big)^{\!\!\ast}\big(x\big)\;\geq\nu\;\bigg\}.
\end{equation}
Observing now that the super-level sets of a non-increasing function are intervals starting from the origin, \eqref{x_nu_hat_hat} becomes
\begin{equation}
x_\nu(\widehat{\varphi})\;=\;\sigma(\varphi)-\meas\Big\{\Big(\varphi.\mathbbm{1}_{\{\varphi=\varphi^\dag\}}\Big)^{\!\!\ast}\geq\nu\Big\}.
\end{equation}
We now use the fact that the Schwarz rearrangement preserves the measure of the super-level sets to obtain
\begin{equation}\label{56}
x_\nu(\widehat{\varphi})\;=\;\sigma(\varphi)-\meas\Big\{\varphi.\mathbbm{1}_{\{\varphi=\varphi^\dag\}}\geq\nu\Big\}.
\end{equation}
We observe that $\varphi<\nu$ on $[0, x_\nu(\varphi)[$ which implies that we have the set equality $$\big\{\varphi.\mathbbm{1}_{\{\varphi=\varphi^\dag\}}\geq\nu\big\} = \big[x_\nu(\varphi),s(\varphi)\big]\setminus\big\{\varphi\neq\varphi^\dag\big\}.$$ We replace this set equality in \eqref{56} and we obtain
\begin{align*}
x_\nu(\widehat{\varphi})&=\sigma(\varphi)-\meas\Big(\big[x_\nu(\varphi),s(\varphi)\big]\setminus\big\{\varphi\neq\varphi^\dag\big\}\Big)\\
&=\sigma(\varphi)-\bigg(s(\varphi)-x_\nu(\varphi)-\meas\Big(\{\varphi\neq\varphi^\dag\}\cap\big[x_\nu(\varphi),\;s(\varphi)\big]\Big)\bigg)
\end{align*}
Replacing $\sigma(\varphi)$ by its expression given at \eqref{def_alpha} leads to
\begin{align*}
x_\nu(\widehat{\varphi})&=x_\nu(\varphi)+\meas\bigg(\{\varphi\neq\varphi^\dag\}\cap\big[x_\nu(\varphi),\;s(\varphi)\big]\bigg)-\meas\bigg(\{\varphi\neq\varphi^\dag\}\cap\big[0,\;s(\varphi)\big]\bigg)\\
&=x_\nu(\varphi)-\meas\Big(\{\varphi\neq\varphi^\dag\}\cap\big[0,\;x_\nu(\varphi)\big]\Big).
\end{align*}
Since $x_\nu(\varphi)\leq s(\varphi)$, we can use the results of Claims 1 and 2 to get
\begin{equation}
x_\nu(\widehat{\varphi})=x_\nu(\varphi)-\meas\Big(\cH_\infty(\varphi)\cap\big[0,\;x_\nu(\varphi)\big]\Big)\;=\;x_\nu(\varphi^\natural).
\end{equation}
Claim 3 is proved.

~\newline\emph{\textbf{Claim 4: The measure of the super-level sets is the same \mathversion{bold}\begin{equation}\meas\{\varphi^\natural\geq\nu\}=\meas\{\widehat{\varphi}\geq\nu\}.\end{equation}}}

The tamping is a rearrangement and so $\meas\{\varphi^\natural\geq\nu\}=\meas\{\varphi\geq\nu\}$. Therefore it is enough to prove that $\meas\{\widehat{\varphi}\geq\nu\}=\meas\{\varphi\geq\nu\}$. The fact that the Schwarz rearrangement preserves the measure of the super-level sets gives
\begin{align*}
\meas\Big(\{\widehat{\varphi}\geq\nu\}\Big)&=\meas\Big(\{\widehat{\varphi}\geq\nu\}\cap[0,\sigma(\varphi)]\Big)+\meas\Big(\{\widehat{\varphi}\geq\nu\}\cap[\sigma(\varphi),+\infty[\Big)\\
&=\meas\bigg(\Big\{\Big(\varphi.\mathbbm{1}_{\{\varphi=\varphi^\dag\}}\Big)^{\!\ast}\geq\nu\Big\}\bigg)\;+\;\meas\bigg(\Big\{\Big(\varphi.\mathbbm{1}_{\{\varphi\neq\varphi^\dag\}}\Big)^{\!\ast}\geq\nu\Big\}\bigg)\\
&=\meas\bigg(\Big\{\varphi.\mathbbm{1}_{\{\varphi=\varphi^\dag\}}\geq\nu\Big\}\bigg)\;+\;\meas\bigg(\Big\{\varphi.\mathbbm{1}_{\{\varphi\neq\varphi^\dag\}}\Big)\geq\nu\Big\}\bigg)\\
&=\meas\Big(\{\varphi\geq\nu\}\Big).
\end{align*}
Claim 4 is proved.

~\newline\emph{\textbf{Conclusion of the proof}}

Let $\nu>0$. By definition of the tamping (Definition \ref{definition_tamping_2}) and by definition of $\widehat{\varphi}$, the two sets $\{\varphi^\natural\geq\nu\}$ and $\{\widehat{\varphi}\geq\nu\}$ are segments. By Claim 4 these two segments have the same length and by Claim 3 the inferior extremity of these two segments coincide. We infer that
$$\{\widehat{\varphi}\geq\nu\}\;=\;\{\varphi^\natural\geq\nu\}.$$
Therefore $\widehat{\varphi}=\varphi^\natural$ almost everywhere.\qed

\subsection{Proof of Lemma \ref{crucial_inequality}}
Let $\Omega$ be a domain of $\RR^d$, let $p\in[1,+\infty[$ and let $\varphi,\psi\in L^p_+(\Omega)$.
The idea behind the inequality stated in Lemma \ref{crucial_inequality} is the observation that the set ${hypo}(\varphi)\triangle{hypo}(\psi)$ is a subset of $\Omega\times\RR_+$ and thus it is possible to apply the Schwarz rearrangement to this set (for the last variable). 

Since $p\geq1$, the function $\nu\longmapsto\nu^{p-1}$ is non-decreasing, for every $x\in\Omega$. Therefore, by the Schwarz rearrangement inequality (Property \ref{Schwarz_propr}),
\begin{equation}\label{Schwarz_Delta}
\int_0^\infty\nu^{p-1}\mathbbm{1}_{\big[\big(\mathrm{hypo}(\varphi)\triangle \mathrm{hypo}(\psi)\big)^\ast\;\big]}(x,\nu)\;\mathrm{d}\nu
\;\leq\;\int_0^\infty\nu^{p-1}\mathbbm{1}_{\big[\mathrm{hypo}(\varphi)\triangle \mathrm{hypo}(\psi)\big]}(x,\nu)\;\mathrm{d}\nu.
\end{equation}
The inequality above is an equality if and only if $p=1$ or if the support of the two functions $\psi$ or $\phi$ are disjoint. On the other hand,
by the layer-cake representation \eqref{layercake2} and the Fubini theorem,
\begin{align}
&\|\varphi-\psi\|_{L^p}^p=p\int_0^\infty\nu^{p-1}\meas\{|\varphi-\psi|\geq\nu\}\;\mathrm{d}\nu\\
&=p\int_\Omega\int_0^\infty\nu^{p-1}\mathbbm{1}_{\big[\mathrm{hypo}(|\varphi-\psi|)\big]}(x,\nu)\;\mathrm{d}\nu\;\mathrm{d}x,\label{crucial_ine_1}
\end{align}
where here  $meas$ refers to the $d$-dimensional Lebesgue measure (with $\Omega\subseteq\RR^d$).
Moreover,
\begin{equation}\mathrm{hypo}(|\varphi-\psi|)\;=\;\bigcup_{x\in\Omega}\;\{x\}\times\big[0,\,|\varphi(x)-\psi(x)|\big]\;=\;\big(\mathrm{hypo}(\varphi)\triangle \mathrm{hypo}(\psi)\big)^\ast.\label{hypopo}\end{equation}
Combining \eqref{crucial_ine_1} and \eqref{hypopo} we get
\begin{equation}
\|\varphi-\psi\|_{L^p}^p=p\int_\Omega\int_0^\infty\nu^{p-1}\mathbbm{1}_{\big[\big(\mathrm{hypo}(\varphi)\triangle \mathrm{hypo}(\psi)\big)^\ast\;\big]}(x,\nu)\;\mathrm{d}\nu\;\mathrm{d}x.\label{crucial_ine_2}\end{equation}
With \eqref{Schwarz_Delta} the inequality above becomes
\begin{equation}
\|\varphi-\psi\|_{L^p}^p\leq p\int_\Omega\int_0^\infty\nu^{p-1}\mathbbm{1}_{\big[\mathrm{hypo}(\varphi)\triangle \mathrm{hypo}(\psi)\big]}(x,\nu)\;\mathrm{d}\nu\;\mathrm{d}x
.\label{crucial_ine_3}\end{equation}
We now make a small abuse of notation by identifying, for every $\nu$, the set $\Omega\times\{\nu\}$ with $\Omega$ (and so we do with their subsets and for the computation of lebesgue measure) which makes the manipulations of integrals and measure theory much easier. With another use of the Fubini theorem, this gives,
\begin{align*}
&p\int_\Omega\int_0^\infty\nu^{p-1}\mathbbm{1}_{\big[\mathrm{hypo}(\varphi)\triangle \mathrm{hypo}(\psi)\big]}(x,\nu)\;\mathrm{d}\nu\;\mathrm{d}x\\
&=p\int_{0}^\infty\int_\Omega\nu^{p-1}\mathbbm{1}_{\big[\big(\mathrm{hypo}(\varphi)\;\triangle\;\mathrm{hypo}(\psi)\big)\cap\big(\Omega\times\{\nu\}\big)\big]}(x)\;\mathrm{d}x\;\mathrm{d}\nu\\
&=p\int_0^\infty\nu^{p-1}\meas_{\RR^d}\Big[\big(\mathrm{hypo}(\varphi)\;\triangle\;\mathrm{hypo}(\psi)\big)\cap\big(\Omega\times\{\nu\}\big)\Big]\;\mathrm{d}\nu\\
&=p\int_0^\infty\nu^{p-1}\meas_{\RR^d}\Big[\big(\mathrm{hypo}(\varphi)\cap\Omega\!\times\!\{\nu\}\big)\;\triangle\;\big(\mathrm{hypo}(\psi)\cap\Omega\!\times\!\{\nu\}\big)\Big]\;\mathrm{d}\nu\\
&=p\int_0^\infty\nu^{p-1}\meas_{\RR^d}\Big(\{\varphi\geq\nu\}\triangle\{\psi\geq\nu\}\Big)\;\mathrm{d}\nu.
\end{align*}
This with \eqref{crucial_ine_3} gives the announced inequality.\qed

\subsection{Proof of Lemma \ref{Cv_equiv}}
Let $\Omega$ be a domain of $\RR^d$, let $p\in[1,+\infty[$ and let $\varphi,\psi\in L^p_+(\Omega)$.

~\newline\emph{\textbf{\mathversion{bold}Claim: If $\varphi$ and $\psi$ are bounded and compactly supported, then
\begin{equation}\label{81}\begin{split}
&\qquad\qquad p\int_0^\infty\nu^{p-1}\meas\Big(\{\varphi\geq\nu\}\triangle\{\psi\geq\nu\}\Big)\,\mathrm{d}\nu\;\leq\;2^{p-1}\|\varphi-\psi\|_{L^p}^p\\&+\;p\,2^{p-1}\bigg(\max\Big(\|\varphi\|_{L^\infty},\,\|\psi\|_{L^\infty}\Big)\,\meas\Big(\supp|\varphi-\psi|\Big)^\frac{1}{p}\bigg)^{p-1}\|\varphi-\psi\|_{L^p}.
\end{split}\end{equation}}}
First we observe that
$$\forall\nu\in\Big[\min(\varphi(x),\psi(x)),\;\max(\varphi(x),\psi(x))\Big],\quad\nu\leq\max(\|\varphi\|_{L^\infty},\,\|\psi\|_{L^\infty}).$$
Therefore for all $x\in\Omega$,
\begin{equation}
\int_{\min(\varphi(x),\psi(x))}^{\max(\varphi(x),\psi(x))}\nu^{p-1}\,\mathrm{d}\,\nu\;\leq\;\int_{\max(\|\varphi\|_{L^\infty},\,\|\psi\|_{L^\infty})}^{\max(\|\varphi\|_{L^\infty},\,\|\psi\|_{L^\infty})+\Big(\max(\varphi(x),\psi(x))-\min(\varphi(x),\psi(x))\Big)}\nu^{p-1}\,\mathrm{d}\,\nu.
\end{equation}
We can rewrite this as follows,
\begin{equation}\label{Schwarz_Delta_rev}\begin{split}
&\qquad\qquad\int_0^\infty\nu^{p-1}\mathbbm{1}_{\big[\mathrm{hypo}(\varphi)\triangle \mathrm{hypo}(\psi)\big]}(x,\nu)\;\mathrm{d}\nu.\\
&\leq\;\int_0^\infty\Big(\nu+\max(\|\varphi\|_{L^\infty},\,\|\psi\|_{L^\infty})\Big)^{p-1}\mathbbm{1}_{\big[\big(\mathrm{hypo}(\varphi)\triangle \mathrm{hypo}(\psi)\big)^\ast\;\big]}(x,\nu)\;\mathrm{d}\nu,
\end{split}\end{equation}
where $\ast$ refers to the Schwarz non-increasing rearrangement for the last variable. The above estimate can be understood as being the converse of \eqref{Schwarz_Delta}. We now integrate for the $x$ variable and the Fubini theorem to get
\begin{align}
&p\int_0^\infty\nu^{p-1}\meas\Big(\{\varphi\geq\nu\}\triangle\{\psi\geq\nu\}\Big)\,\mathrm{d}\nu\\
\leq p\int_0^\infty\Big(\nu+\max(\|\varphi\|_{L^\infty}\!,&\|\psi\|_{L^\infty}\!)\!\Big)^{p-1}\!\!\!.\meas\Big(\big(\mathrm{hypo}(\varphi)\triangle\mathrm{hypo}(\psi)\big)^\ast\!\cap(\RR\times\{\nu\})\Big)\mathrm{d}\nu.
\end{align}
We then use in the estimate above the fact that $(a+b)^\alpha\leq2^{\alpha}(a^\alpha+b^\alpha)$ (with $\alpha>0$), the layer-cake representation of the $L^p$ norms and Equality \eqref{hypopo}. This gives,
\begin{equation}\begin{split}\label{grr}
&p\int_0^\infty\nu^{p-1}\meas\Big(\{\varphi\geq\nu\}\triangle\{\psi\geq\nu\}\Big)\,\mathrm{d}\nu\\
\leq\;2^{p-1}\,\|\varphi&-\psi\|_{L^p}^p+p\,2^{p-1}\Big(\max(\|\varphi\|_{L^\infty}\!,\|\psi\|_{L^\infty}\!)\Big)^{p-1}\,\|\varphi-\psi\|_{L^1}.\end{split}
\end{equation}
In the other hand, the Hölder inequality gives 
\begin{equation}\label{holderr}
\|\varphi-\psi\|_{L^1}\;\leq\;\Big(\meas(\supp|\varphi-\psi|)\Big)^\frac{p-1}{p}\,\|\varphi-\psi\|_{L^p}.
\end{equation}
We then obtain \eqref{81} by combining \eqref{grr} and \eqref{holderr}.

~\newline\emph{\textbf{Passing to the limit and conclusion of the proof}}

Consider now $\varphi\in L^p_+(\Omega)$ and a sequence $\varphi_n\in L^p_+(\Omega)$ converging towards $\varphi$ in $L^p$. We note $\cB(x,a)$ the open ball of center $x$ and of radius $a$.
If we apply the inequality given by the claim before to the functions $\varphi.\mathbbm{1}_{\cB(0,b)}.\mathbbm{1}_{\{\varphi\leq a\}}$ and $\varphi_n.\mathbbm{1}_{\cB(0,b)}.\mathbbm{1}_{\{\varphi\leq a\}}$, we obtain a constant $C$ such that
\begin{align}
&\qquad\qquad p\int_0^a\nu^{p-1}\meas\Big((\{\varphi_n\geq\nu\}\triangle\{\varphi\geq\nu\})\cap\cB(0,b)\Big)\,\mathrm{d}\nu\label{89}\\
&\qquad \qquad\qquad\leq\;C\|\varphi_n-\varphi\|_{L^p}^p\;+C\big(a.b^\frac{1}{p}\big)^{p-1}\!.\|\varphi-\varphi_n\|_{L^p}.\label{90}
\end{align}
When $a$ and $b$ go to infinity, the term \eqref{89} converges towards 
\begin{equation}\label{79}p\int_0^{+\infty}\nu^{p-1}\meas(\{\varphi_n\geq\nu\}\triangle\{\varphi\geq\nu\})\,\mathrm{d}\nu\;<\;+\infty.\end{equation}Let $\varepsilon>0$. We fix $a$ and $b$ large enough such that the term \eqref{89} is at distance at most $\varepsilon$ to the term \eqref{79}. Now that $a$ and $b$ are fixed, we use the fact that \eqref{90} worthes less than an $\varepsilon>0$ when $n$ is chosen large enough. The lemma is proved.\qed

\subsection{Proof of Lemma \ref{before_crucial_lemma}}
Let $\Omega$ be a domain of $\RR^d$ and let $p\in(1,+\infty)$.
Let $\varphi\in L^p_+(\Omega)$ and let $\varphi_n\in L^p_+(\Omega)$ such that $M:=\sup_n\|\varphi_n\|_{L^p} <+\infty$ and
$$\meas\Big(\{\varphi_n\geq\nu\}\triangle\{\varphi\geq\nu\}\Big)\;\xrightarrow[n\rightarrow+\infty]{}\;0,$$
for almost every $\nu>0$.
Let $g\in L^q(\Omega)$ where $q$ is the Hölder conjugate of $p$ defined by $\frac{1}{p}+\frac{1}{q}=1$. We want to prove that
\begin{equation}\label{crucial_cv}\int_\Omega g(x)\,(\varphi_n-\varphi)(x)\;\mathrm{d}x\quad\xrightarrow[n\rightarrow\infty]{}\quad0.\end{equation}

~\newline\emph{\textbf{\mathversion{bold} Claim 1: The proof can be reduced to functions $\varphi$ and $\varphi_n$ with support in $\Omega\setminus\cB(0,a)$ with $a$ independent of $n$.}}

Let $a\in\RR_+$. Using the Hölder inequality, we get that
\begin{align*}
\int_{\Omega\setminus\cB(0,a)} &g(x)\,(\varphi_n-\varphi)(x)\;\mathrm{d}x\;\\
&\leq\left(\int_{\Omega\setminus\cB(0,a)}|g(x)|^q\;\mathrm{d}x\right)^\frac{1}{q}\,\left(\int_{\Omega\setminus\cB(0,a)}\big|(\varphi_n-\varphi)(x)\big|^p\;\mathrm{d}x\right)^\frac{1}{p}\\
&\leq\left(\int_{\Omega\setminus\cB(0,a)}|g(x)|^q\;\mathrm{d}x\right)^\frac{1}{q}\,\Bigg(\;\|\varphi_n\|_{L^p}+\|\varphi\|_{L^p}\Bigg)\\
&\leq\left(\int_{\Omega\setminus\cB(0,a)}|g(x)|^q\;\mathrm{d}x\right)^\frac{1}{q}\,\Bigg(\;M+\|\varphi\|_{L^p}\Bigg).
\end{align*}
The above estimate does not depend on $n$ and it is smaller than a fixed $\varepsilon>0$ if $a$ is chosen large enough because we have $q<\infty$. Then, in order to prove \eqref{crucial_cv}, it is enough to consider only uniformly compactly supported functions. Otherwise, it is always possible to separate the integral appearing in \eqref{crucial_cv} into two - an integral on $\Omega\cap\cB(0,a)$ and an integral on $\Omega\setminus\cB(0,a)$ - and to prove that they are both smaller than any $\varepsilon>0$ when the parameters $a$ and $n$ are well chosen. 

~\newline\emph{\textbf{\mathversion{bold} Claim 2: The proof can be reduced to the case $\varphi\in L^\infty(\Omega)$.}}

Let $b\in\RR_+$. The reasoning is almost the same as the step before.
\begin{align*}
\int_{\{\varphi\geq b\}} &g(x)\,(\varphi_n-\varphi)(x)\;\mathrm{d}x\;\\
&\leq\left(\int_{\{\varphi\geq b\}}|g(x)|^q\;\mathrm{d}x\right)^\frac{1}{q}\,\left(\int_{\{\varphi\geq b\}}|(\varphi_n-\varphi)(x)|^p\;\mathrm{d}x\right)^\frac{1}{p}\\
&\leq\left(\int_{\{\varphi\geq b\}}|g(x)|^q\;\mathrm{d}x\right)^\frac{1}{q}\,\Bigg(\;M+\|\varphi\|_{L^p}\Bigg).
\end{align*}
The above estimate does not depend on $n$ and it is smaller than a fixed $\varepsilon>0$ if $b$ is chosen large enough. Then, in order to prove \eqref{crucial_cv}, we can suppose that the function $\varphi$ is essentially bounded. 

~\newline\emph{\textbf{\mathversion{bold} Claim 3: The proof can be reduced to the case $g\in L^\infty$.}}

The same reasoning as for Claim 2 with $g$ instead of $\varphi$ leads to the same estimate (mutatis mutandis).

~\newline\emph{\textbf{\mathversion{bold} Claim 4: The proof can be reduced to the case $\;
\sup\limits_{n\in\NN}\;\|\varphi_n\|_{L^\infty}\;<\;+\infty.
$}}

This reasoning is also a variation of the two reasonings before. We have
\begin{equation}\label{102}
\int_{\{\varphi_n\geq b\}} g(x).(\varphi_n-\varphi)(x)\;\mathrm{d}x\;\leq\left(\int_{\{\varphi_n\geq b\}}|g(x)|^q\;\mathrm{d}x\right)^\frac{1}{q}\,\Bigg(\;M+\|\varphi\|_{L^p}\Bigg)\end{equation}
For any sets $A$ and $B$ we always have $A\subseteq B\cup(A\triangle B)$. We use this fact in equation \eqref{102} and we obtain
\begin{align*}
&\qquad\qquad\int_{\{\varphi_n\geq b\}} g(x).(\varphi_n-\varphi)(x)\;\mathrm{d}x\;\\
&\leq\Bigg(\;M+\|\varphi\|_{L^p}\Bigg)
\left[\int_{\{\varphi\geq b\}}|g(x)|^q\;\mathrm{d}x+\int_{\{\varphi\geq b\}\triangle\{\varphi_n\geq b\}}|g(x)|^q\;\mathrm{d}x\right]^\frac{1}{q}\\
&\leq\Bigg(\;M+\|\varphi\|_{L^p}\Bigg)\Bigg[\left(\int_{\{\varphi\geq b\}}|g(x)|^q\;\mathrm{d}x\right)^\frac{1}{q}+\|g\|_{L^\infty}.\bigg(\meas\Big(\{\varphi\geq b\}\triangle\{\varphi_n\geq b\}\Big)\bigg)^{\frac{1}{q}}\Bigg]\\
&=\;\Bigg(\;M+\|\varphi\|_{L^p}\Bigg)\;\|g\|_{L^\infty}\,\bigg(\meas\Big(\{\varphi\geq b\}\triangle\{\varphi_n\geq b\}\Big)\bigg)^{\frac{1}{q}},
\end{align*}
where we used Claim 2 that states $\{\varphi\geq b\}$ is of measure $0$ if $b$ is chosen large enough. Now we use the hypothesis
$$\meas\bigg(\{\varphi\geq b\}\triangle\{\varphi_n\geq b\}\bigg)\;\xrightarrow[n\rightarrow\infty]{}\;0,$$
and we conclude that the functions $\varphi_n$ can be taken essentially bounded independently of $n$.

~\newline\emph{\textbf{Conclusion of the proof.}}

It is now supposed that the functions we manipulate are all essentially bounded by the same bound $b\in\RR_+$ and have a compact support contained in a ball $\cB(0,a)$ with $a\in\RR_+$ independent of $n$. If we combine the Hölder inequality and the inequality given by Lemma \ref{crucial_inequality} we get
\begin{align*}
&\int_\RR g(x)\,(\varphi_n-\varphi)(x)\;\mathrm{d}x\;\\\leq\;\left(\int_\RR|g(x)|^q\mathrm{d}x\right)^\frac{1}{q}&\,\left(p\int_0^\infty\nu^{p-1}\meas\Bigg(\{\varphi_n\geq\nu\}\triangle\{\varphi\geq\nu\}\Bigg)\mathrm{d}\nu\right)^\frac{1}{p}.
\end{align*}
We conclude the proof with the Lebesgue dominated convergence theorem.\qed

\subsection{Proof of Lemma \ref{crucial_lemma} for $p=1$}
Let $\Omega$ be a domain of $\RR^d$. Let $\varphi\in L^1_+(\Omega)$ and let $\varphi_n\in L^1_+(\Omega)$ such that
\begin{equation}\|\varphi_n\|_{L^1}\longrightarrow\|\varphi\|_{L^1}\end{equation} and
\begin{equation} \meas\Big(\{\varphi_n\geq\nu\}\triangle\{\varphi\geq\nu\}\Big)\;\xrightarrow[n\rightarrow+\infty]{}\;0,\end{equation}
for almost every $\nu$. We want to prove that 
\begin{equation}
\|\varphi_n-\varphi\|_{L^1}\;\xrightarrow[n\rightarrow\infty]\;0.\end{equation}
Let $a$ and $b$ in $\RR_+$. The inequality provided at Lemma \ref{crucial_inequality} gives
\begin{align*}
&\int_{\cB(0,a)}\Big|\varphi.\mathbbm{1}_{\{\varphi\leq b\}}-\varphi_n.\mathbbm{1}_{\{\varphi_n\leq b\}}\Big|(x)\;\mathrm{d}x\\
\leq\int_0^b&\meas\Big((\{\varphi\geq\nu\}\triangle\{\varphi_n\geq\nu\})\;\cap\;\cB(0,a)\Big)\;\mathrm{d}\nu.\end{align*}
Since the measure of the symmetrical difference of the super-level sets is vanishing when $n$ grows, the Lebesgue dominated convergence theorem with the above estimate gives
\begin{equation}\label{local_crucial_cv}\int_{\cB(0,a)}\Big|\varphi.\mathbbm{1}_{\{\varphi\leq b\}}-\varphi_n.\mathbbm{1}_{\{\varphi_n\leq b\}}\Big|(x)\;\mathrm{d}x\;\xrightarrow[n\rightarrow\infty]{}\;0.\end{equation}
We now define
$$R_{a,b}(\varphi)\;:=\;\varphi-\varphi.\mathbbm{1}_{\{\varphi\leq b\}}.\mathbbm{1}_{\cB(0,a)}.$$
Since we work with $L^1$ norms, we have
$$\|R_{a,b}(\varphi_n)\|_{L^1}\;=\;\|\varphi_n\|_{L^1}\;-\;\big\|\varphi_n.\mathbbm{1}_{\{\varphi_n\leq b\}}.\mathbbm{1}_{\cB(0,a)}\big\|_{L^1}.$$
We use the hypothesis $\|\varphi_n\|_{L^1}\longrightarrow\|\varphi\|_{L^1}$ and the convergence result \eqref{local_crucial_cv} to obtain
\begin{equation}\label{local_crucial_rest}\|R_{a,b}(\varphi_n)\|_{L^1}\;\xrightarrow[n\rightarrow\infty]{}\;\|\varphi\|_{L^1}\;-\;\big\|\varphi.\mathbbm{1}_{\{\varphi\leq b\}}.\mathbbm{1}_{\cB(0,a)}\big\|_{L^1}\;=\;\|R_{a,b}(\varphi)\|_{L^1}.\end{equation}
We now write
\begin{align*}\|\varphi-\varphi_n\|_{L^1}\;\leq\;\big\|\varphi_n.&\mathbbm{1}_{\{\varphi_n\leq b\}}.\mathbbm{1}_{\cB(0,a)}-\varphi.\mathbbm{1}_{\{\varphi\leq b\}}.\mathbbm{1}_{\cB(0,a)}\big\|_{L^1}\\
&\;+\;\|R_{a,b}(\varphi_n)\|_{L^1}\;+\;\|R_{a,b}(\varphi)\|_{L^1}.
\end{align*}
The term $\|R_{a,b}(\varphi)\|_{L^1}$ can be made smaller than a given $\varepsilon$ by choosing $a$ and $b$ large enough. By \eqref{local_crucial_rest}, the term $\|R_{a,b}(\varphi_n)\|_{L^1}$ can be made $\varepsilon$ close to $\|R_{a,b}(\varphi)\|_{L^1}$ (and thus $2\varepsilon$ close to $0$) by choosing $n$ large enough. The first term of the above inequality is also smaller than a given $\varepsilon$ when $n$ is chosen large enough by the convergence result \eqref{local_crucial_cv}. Therefore,
$$\|\varphi-\varphi_n\|_{L^1}\;\xrightarrow[n\rightarrow\infty]{}\;0.$$\qed

\subsection{Proof of Lemma \ref{compactness_result}}
Let $p\in[1,+\infty)$. Let $(\varphi_n)\in L^p_+(\RR_+)$ be a bounded sequence such that
\begin{equation}
\exists\;\mu>0,\qquad\limsup\limits_{n\to+\infty}\;x_\mu(\varphi_n)\;<\;+\infty,
\end{equation}
where $x_\mu(\varphi_n)$ is the essential infimum of the super-level set at level $\mu$ of function $\varphi_n$. We want to prove that up to an omitted extraction there exists a function $\psi$ such that for almost every $\nu>0$
\begin{equation}
\meas\Big(\{\varphi_n^\natural\geq\nu\}\;\triangle\;\{\psi\geq\nu\}\Big)\;\xrightarrow[n\rightarrow+\infty]{}\;0.\label{115}
\end{equation}

~\newline\emph{\textbf{\mathversion{bold} Claim 1: For every $\nu>0$ the sequence $x_\nu(\varphi_n^\natural)$ is a bounded sequence}}\newline
To prove this claim we use the hypothesis that
\begin{equation}
\exists\;\mu>0,\qquad\limsup\limits_{n\to+\infty}\;x_\mu(\varphi_n)\;<\;+\infty.
\end{equation}
Since $\nu\leq\mu\Rightarrow x_\nu(\varphi)\leq x_\mu(\varphi)$, the case $\nu\leq\mu$ is straight-forward. Concerning the case $\nu\geq\mu$, we first use the fact that, by inclusion of the super-level sets,
\begin{equation}\label{122}
x_\nu(\varphi_n^\natural)\;\leq\;\supess\{\varphi_n^\natural\geq\mu\}.
\end{equation}
Since the super-level sets of $\varphi_n^\natural$ are segments and since $x_\mu(\varphi_n^\natural)$ remains bounded, the quantity appearing at the right-hand side of \eqref{122} remains bounded as $n\to\infty$ if and only if the measure of $\{\varphi_n^\natural\geq\mu\}$ remains bounded as $n\to\infty$. Nevertheless, by hypothesis $(\varphi_n)$ is bounded in $L^p$. Since the tamping is a rearrangement, we have for all $\lambda>0$
\begin{equation}\label{123}
\lambda^p\,\meas\{\varphi_n^\natural\geq\lambda\}\;\leq\;\|\varphi_n^\natural\|_{L^p}
=\|\varphi_n\|_{L^p}\;\leq\;\limsup\limits_{n\to+\infty}\|\varphi_n\|_{L^p}\;<\;+\infty.
\end{equation}
This implies in particular that $meas\{\varphi_n^\natural\geq\lambda\}$ is a bounded sequence. Therefore, the case $\lambda=\mu$ in \eqref{123} combined with \eqref{122} give the claim.

~\newline\emph{\textbf{\mathversion{bold} Claim 2: There exist an extraction $\sigma$ and a set $A\subseteq\RR_+^\ast$ countable and dense such that that there exists a decreasing family of segments $(T_\nu)_{\nu\in A}$ verifying the convergence property
\begin{equation}
\forall \nu\in A,\qquad\meas\Big(\Big\{\varphi_{\sigma(n)}^\natural\;\geq\;\nu\Big\}\;\triangle\;T_\nu\Big)\;\xrightarrow[n\to+\infty]{}0.
\end{equation}}}\newline
We first recall that $\{\varphi_{n}^\natural\;\geq\;\nu\}$ is also a segment.
By Claim 1, the respective lower extremity $x_\nu(\varphi_n^\natural)$ of the segments $\{\varphi_{n}^\natural\;\geq\;\nu\}$ remains bounded as $n\to+\infty$. Using again \eqref{123} we obtain that the length of these segments remains bounded as $n\to+\infty$ and therefore it is also the case for their upper extremity. By the Bolzano-Weierstrass theorem, we have for all $\nu>0$ an extraction $\sigma_\nu$ such that the two extremities of the segment $\{\varphi^\natural_n\geq\nu\}$ converge as $n\to+\infty$. The claim follows from a classical diagonal extraction argument.

~\newline\emph{\textbf{\mathversion{bold} Conclusion of the proof.}}\newline
To make the arguments easier to read, we now omit the extraction. The natural candidate for the limit is the function $\psi$ defined by
\begin{equation}
\psi(x)\;:=\;\sup\;\{\nu\in A\;/\;x\in T_\nu\}.
\end{equation}
This definition implies in particular that the super-level sets of $\psi$ are segments.
The result of Claim 2 gives \eqref{115} but for $\nu\in A$ which is only a countable set and then we cannot conclude yet. Observe first that the set of $\nu>0$ such that the measure of $\{\psi=\nu\}$ is $0$ is a countable set and since the announced result \eqref{115} holds for almost every $\nu>0$ we can exclude this case. Now, let $\nu\in\RR_+^\ast\setminus A$ such that $meas\{\psi=\nu\}=0$. Proving that \eqref{115} holds for this choice $\nu$ concludes the proof.
For that purpose we prove that the two extremities of the segment $\{\varphi_n^\natural\geq\nu\}$ are converging towards the corresponding extremities of the segment $\{\psi\geq\nu\}$.

We first claim that if $\nu\notin A$ is such that there exists $(\nu_k)\in A$ converging towards $\nu$ while $x_{\nu_k}(\psi)$ is not converging towards $x_\nu(\psi)$ then the measure of $\{\psi=\nu\}$ is not $0$ (which is the excluded case). Indeed, suppose for instance that $\lambda$ converges towards $\nu$ by upper values. By inclusion of the super-level sets, the sequence $x_{\lambda}(\psi)$ is non-increasing and then converges. Since the super-level sets of $\psi$ are segments, 
\begin{equation}
\forall\;x\in\left[x_\nu(\psi),\lim\limits_{\lambda\to\nu^+}x_{\lambda}(\psi)\right],\qquad\psi(x)=\nu.
\end{equation}
The same reasonning works if $\lambda$ converges towards $\nu$ by lower values. 
Since the set of $\nu$ such that $meas\{\psi=\nu\}\neq0$ is countable and since we want to establish that \eqref{115} holds for almost every $\nu>0$, we only have to concentrate on the case
\begin{equation}\label{areuh}
x_\lambda(\psi)\;\xrightarrow[\lambda\to\nu]{}\;x_\nu(\psi),\qquad \mathrm{with\;}\lambda\in A.
\end{equation}
We consider $\mu\leq\nu\leq\lambda$ with $\mu$ and $\lambda$ in $A$. By inclusion of the super level-sets, 
\begin{equation}
\liminf\limits_{n\to+\infty}x_\mu(\varphi_n^\natural)\;\leq\;\liminf\limits_{n\to+\infty}x_\nu(\varphi_n^\natural)\;\leq\;\limsup\limits_{n\to+\infty}x_\nu(\varphi_n^\natural)\;\leq\;\limsup\limits_{n\to+\infty}x_\lambda(\varphi_n^\natural).
\end{equation}
Using now Claim 2. we are led to
\begin{equation}
x_\mu(\psi)\;\leq\;\liminf\limits_{n\to+\infty}x_\nu(\varphi_n^\natural)\;\leq\;\limsup\limits_{n\to+\infty}x_\nu(\varphi_n^\natural)\;\leq\;x_\lambda(\psi).
\end{equation}
Making $\mu\to\nu^-$ and $\lambda\to\nu^+$ and using \eqref{areuh} we get
\begin{equation}
\lim\limits_{n\to+\infty}x_\nu(\varphi^\natural_n)\;=\;x_\nu(\psi).
\end{equation}
The same reasonning works for the upper extremity of the segment $\{\varphi_{n}^\natural\;\geq\;\nu\}$ and thus \eqref{115} holds. The theorem is proved.\qed

\subsection{Proof of Theorem \ref{continuity_theorem_1D}}
We recall that the definitions of $x_\nu$ and $y_\nu$ were given at Proposition \ref{L^p_tamping}.
Since the super-level sets of $\varphi$ that are of measure zero create no problem for the convergence of $\varphi_n^\natural$, we will make the assumption that, without loss of generality, 
$$\forall \nu,\quad\meas\{\varphi\geq\nu\}>0.$$
This is equivalent to
\begin{equation}\label{x_nu_fini}
\forall \nu,\quad x_\nu(\varphi)<\infty\qquad\text{and}\qquad\limsup\limits_{n\rightarrow+\infty}x_\nu(\varphi_n)<\infty.
\end{equation}
The strategy of the proof consists in proving that the condition $\varphi_n\rightarrow\varphi$ in $L^p$, the condition \eqref{46} on the convergence of the super-level sets and the \emph{local hollows convergence condition},
\begin{equation}\label{hollows_convergence}
\forall a>0,\qquad\meas\Big((\cH_\infty(\varphi_n)\triangle\cH_\infty(\varphi))\cap[0,a]\Big)\;\xrightarrow[n\rightarrow+\infty]{}\;0,
\end{equation}
together imply that for almost every $\nu\geq0$
\begin{equation}
\meas\Big(\{\varphi^\natural_n\geq\nu\}\triangle\{\varphi^\natural\geq\nu\}\Big)\;\xrightarrow[n\rightarrow+\infty]{}\;0.
\end{equation}
In this case the hypothesis of Lemma \ref{crucial_lemma} are satisfied and we can conclude the proof. For that purpose, we prove that for almost every $\nu$
$$y_\nu(\varphi_n)\;\xrightarrow[n\rightarrow+\infty]{}\;y_\nu(\varphi).$$
We start by defining the segment
$$X_\nu\left(\varphi_n,\varphi\right)\;:=\;\left[\liminf\limits_{n\rightarrow+\infty} x_\nu(\varphi_n),\;x_\nu(\varphi)\right].$$
The fact that the extreminites of this segement are well-ordered, meaning that $\liminf\limits_{n\rightarrow+\infty} x_\nu(\varphi_n)\!~\leq~\! x_\nu(\varphi)$, is a consequence of Claim 1 below.

~\newline\emph{\textbf{\mathversion{bold} Claim 1: $\limsup\limits_{n\rightarrow+\infty} x_\nu(\varphi_n)\in X_\nu\left(\varphi_n,\varphi\right)$.}}\newline
By the absurd, suppose that $\limsup\limits_{n\rightarrow+\infty} x_\nu(\varphi_n)>x_\nu(\varphi)$.
The definition of $x_\nu(\varphi_n)$ gives
\begin{equation}\label{117}
\{\varphi_n\leq\nu\}\cap[0,x_\nu(\varphi_n)]\;=\;[0,x_\nu(\varphi_n)].
\end{equation}
By \eqref{117} we obtain
\begin{equation}
\int_{[x_\nu(\varphi),x_\nu(\varphi_n)]\cap\{\varphi\geq\nu\}}|\varphi-\nu|^p\;\leq\;\int_{[x_\nu(\varphi),x_\nu(\varphi_n)]\cap\{\varphi\geq\nu\}}|\varphi-\varphi_n|^p\;\leq\;\int_{\RR_+}|\varphi-\varphi_n|^p.
\end{equation}
and therefore
\begin{equation}\label{118}
\limsup\limits_{n\to\infty}\int_{[x_\nu(\varphi),x_\nu(\varphi_n)]\cap\{\varphi\geq\nu\}}|\varphi-\nu|^p\;\leq\;\lim\limits_{n\to\infty}\int_{\RR_+}|\varphi-\varphi_n|^p=0.
\end{equation}
By definition of $x_\nu(\varphi)$ we have that for all $\varepsilon>0$,
\begin{equation}\label{116}
\meas\Big(\{\varphi>\nu\}\cap[x_\nu(\varphi),x_\nu(\varphi)+\varepsilon]\Big)\;>\;0.
\end{equation}
The fact that  $\limsup\limits_{n\rightarrow+\infty} x_\nu(\varphi_n)>x_\nu(\varphi)$ with \eqref{116} gives
\begin{equation}\label{120}
\limsup\limits_{n\to\infty}\int_{[x_\nu(\varphi),x_\nu(\varphi_n)]}|\varphi-\nu|^p\;>\;0.
\end{equation}
Equations \eqref{118} and \eqref{120} are in contradiction.

~\newline\emph{\textbf{\mathversion{bold} Claim 2: We have $\varphi\leq\nu$ almost everywhere on $X_\nu(\varphi_n,\varphi)$.}}

\noindent If we have $\varphi>\nu$ on a subset $A$ of $X_\nu(\varphi_n,\varphi)$ whose measure is non zero then given the definition of $x_\nu$ and the fact that $\varphi_n$ converges in $L^p$ towards $\varphi$, holds
$$x_\nu(\varphi)\;\leq\;\mathrm{inf\;ess}\; A\qquad\mathrm{and}\qquad\limsup\limits_{n\rightarrow+\infty} x_\nu(\varphi_n)\;\leq\;\mathrm{inf\;ess}\; A.$$
This is in contradiction with the fact that, by definition of $X_\nu\left(\varphi_n,\varphi\right)$,
$$\infess A\;<\;\supess A\;\leq\;\max\; X_\nu\left(\varphi_n,\varphi\right)$$
because the maximum of $X_\nu\left(\varphi_n,\varphi\right)$ is $x_\nu(\varphi)$.

~\newline\emph{\textbf{\mathversion{bold} Claim 3: We have $\varphi\geq\nu$ almost everywhere on $X_\nu(\varphi_n,\varphi)\setminus\cH_\infty(\varphi)$.}}

\noindent Suppose that we have $\varphi\leq\nu-\varepsilon$ on a subset $A$ of $X_\nu(\varphi_n,\varphi)$ with $A$ of positive measure (for a certain $\varepsilon>0$).
Since by Claim 1 $\displaystyle\limsup\limits_{n\rightarrow+\infty} x_\nu(\varphi_n)\leq x_\nu(\varphi),$ then
\begin{equation}\liminf\limits_{n\rightarrow+\infty} x_\nu(\varphi_n)\;\leq\;\mathrm{inf\;ess}\; A.\end{equation}
We assumed at \eqref{x_nu_fini} that $x_\nu(\varphi)<\infty$, then we know by $L^p$ convergence that there exists a set $B$ whose essential infimum is greater than $x_\nu(\varphi)$ and such that, for $n$ large enough, ${\varphi_n}_{|B}\geq\nu-\varepsilon/2$. Such a set verifies
\begin{equation}\label{X_nu_hollows_1}\mathrm{sup\;ess\;}A\leq\mathrm{inf\;ess\;}B.\end{equation}

Moreover, by definition of $x_\nu(\varphi_n)$, there exist sets $C_n$ of positive measure and an extraction $\sigma$ such that 
\begin{equation}\label{X_nu_hollows_2}\mathrm{inf\;ess}\; A\geq\liminf\limits_{n\rightarrow+\infty}\left(\;\mathrm{inf\;ess\;}C_n\;\right)\end{equation}
and ${\varphi_{\sigma(n)}}\geq\nu$ on $C_n$. The extraction $\sigma$ is an extraction such that $x_\nu(\varphi_{\sigma(n)})\rightarrow\liminf\limits_{n\rightarrow+\infty} x_\nu(\varphi_n)$. Such an extraction exists by property of the ``liminf''. Both \eqref{X_nu_hollows_1} and \eqref{X_nu_hollows_2} imply that
\begin{equation}\label{inclusion_A}A\subseteq\left[\liminf\limits_{n\rightarrow+\infty}\left(\;\mathrm{inf\;ess\;}C_n\;\right),\;\mathrm{inf\;ess\;}B\right]\setminus\left\{\varphi_n\geq\nu-\frac{\varepsilon}{2}\right\}.\end{equation}
This inclusion is true for almost every element of $A$. Nevertheless, by definition of the \emph{hollows} at \eqref{define_hollows}, we have for all measurable set $M$
$$\mathrm{conv\;ess\;}M\subseteq\conv M\qquad\text{and}\qquad\mathrm{conv\;ess\;}M\setminus M\subseteq\cH(M),$$
Therefore, regarding the definitions of the sets $B$ and $C_n$,
\begin{equation}\label{inclusion_B_n}
\left[\mathrm{inf\;ess\;}C_n\;,\;\mathrm{inf\;ess\;}B\right]\setminus\left\{\varphi_n\geq\nu-\frac{\varepsilon}{2}\right\}\subseteq\cH_{\nu-\varepsilon/2}(\varphi_n)\subseteq\cH_\infty(\varphi_n).
\end{equation}
This inclusion is true for almost every element. Since we demanded that the sequence $\varphi_n$ verifies the convergence of the hollows \eqref{hollows_convergence}, we conclude that the inclusion \eqref{inclusion_B_n} is preserved asymptotically, which means that
\begin{equation}\label{inclusion_B}
\left[\liminf\limits_{n\rightarrow+\infty}\left(\;\mathrm{inf\;ess\;}C_n\;\right)\;,\;\mathrm{inf\;ess\;}B\right]\setminus\left\{\varphi\geq\nu-\frac{\varepsilon}{2}\right\}\subseteq\cH_\infty(\varphi).
\end{equation}
We conclude from \eqref{inclusion_A} and \eqref{inclusion_B} that $A\subseteq\cH_\infty(\varphi)$ for almost every element of $A$. This gives the conclusion.

~\newline\emph{\textbf{\mathversion{bold} Claim 4: If we have $\meas\{\varphi=\nu\}=0$ then,
$\;X_\nu(\varphi,\varphi_n)\subseteq \cH_\infty(\varphi).$}}

The inclusion is true up to a set of measure zero.

By Claim 2 we have, $\varphi_{|X_\nu(\varphi,\varphi_n)}\leq\nu$ almost everywhere on $X_\nu(\varphi,\varphi_n).$ But since $\meas\{\varphi=\nu\}=0$, we can conclude that $\varphi_{|X_\nu(\varphi,\varphi_n)}<\nu$ almost everywhere. Claim 3 gives the announced inclusion.

~\newline\emph{\textbf{\mathversion{bold} Claim 5: If we have $\meas\{\varphi=\nu\}=0$ then $y_\nu(\varphi_n)\xrightarrow[n\rightarrow+\infty]{}y_\nu(\varphi).$}}

We have :
\begin{equation}\begin{split}
y_\nu&(\varphi_n)=x_\nu(\varphi_n)-\meas\bigg(\cH_\infty(\varphi_n)\cap[0,x_\nu(\varphi_n)]\bigg)\\
&=y_\nu(\varphi)+x_\nu(\varphi_n)-x_\nu(\varphi)\\
&+\meas\bigg(\cH_\infty(\varphi)\cap[0,x_\nu(\varphi)]\bigg)-\meas\bigg(\cH_\infty(\varphi)\cap[0,x_\nu(\varphi_n)]\bigg)\\
&+\meas\bigg(\cH_\infty(\varphi)\cap[0,x_\nu(\varphi_n)]\bigg)-\meas\bigg(\cH_\infty(\varphi_n)\cap[0,x_\nu(\varphi_n)]\bigg).\end{split}
\end{equation}
Thus,
\begin{align}&\left|y_\nu(\varphi_n)-y_\nu(\varphi)\right|\;\leq\;\bigg|\bigg(x_\nu(\varphi_n)-x_\nu(\varphi)\bigg)\label{103}\\
&\qquad-\bigg(\meas\bigg(\cH_\infty(\varphi)\cap[0,x_\nu(\varphi_n)]\bigg)-\meas\bigg(\cH_\infty(\varphi)\cap[0,x_\nu(\varphi)]\bigg)\bigg)\bigg|\label{104}\\
&+\bigg|\meas\bigg(\cH_\infty(\varphi)\cap[0,x_\nu(\varphi_n)]\bigg)-\meas\bigg(\cH_\infty(\varphi_n)\cap[0,x_\nu(\varphi_n)]\bigg)\bigg|\label{105}
\end{align}
We will explain why the two terms of the sum in the right hand side of the above inequality are vanishing when $n$ goes to $\infty$.

Concerning \eqref{105}, we always have
\begin{align}
\bigg|\meas&\bigg(\cH_\infty(\varphi)\cap[0,x_\nu(\varphi_n)]\bigg)-\meas\bigg(\cH_\infty(\varphi_n)\cap[0,x_\nu(\varphi_n)]\bigg)\bigg|\\
&\leq\meas\bigg[\bigg(\cH_\infty(\varphi)\cap[0,x_\nu(\varphi_n)]\bigg)\triangle\bigg(\cH_\infty(\varphi_n)\cap[0,x_\nu(\varphi_n)]\bigg)\bigg]\\
&=\meas\bigg[\bigg(\cH_\infty(\varphi)\triangle\cH_\infty(\varphi_n)\bigg)\cap[0,x_\nu(\varphi_n)]\bigg]\\
&\leq\meas\bigg(\cH_\infty(\varphi)\triangle\cH_\infty(\varphi_n)\bigg),
\end{align}
which is converging towards zero when $n$ goes to infinity by hypothesis \eqref{hollows_convergence}.

Concerning \eqref{103}-\eqref{104}, as we explain here after, the result is a consequence of the properties established at the previous claims.
\begin{itemize}
\item\emph{Case 1: $x_\nu(\varphi_n)\in\left[\liminf\limits_{n\rightarrow+\infty} x_\nu(\varphi_n),\;\limsup\limits_{n\rightarrow+\infty} x_\nu(\varphi_n)\right]$.}

The result established at Claim 4 and the fact that $X_\nu(\varphi,\varphi_n)$ is a segment imply that
\begin{equation}\cH_\infty(\varphi)\cap\conv\left\{x_\nu(\varphi_n),x_\nu(\varphi)\right\}=\conv\left\{x_\nu(\varphi_n),x_\nu(\varphi)\right\}.\end{equation}
Thus,
\begin{align}\meas\bigg(\cH_\infty(\varphi)\cap[0,x_\nu(\varphi)]\bigg)-\meas\bigg(\cH_\infty(\varphi)\cap[0,x_\nu(\varphi_n)]\;\bigg)=x_\nu(\varphi)-x_\nu(\varphi_n),\label{with_step_4}\end{align}
and then the studied term is $0$.

\item\emph{Case 2: $x_\nu(\varphi_n)\notin\left[\liminf\limits_{n\rightarrow+\infty} x_\nu(\varphi_n),\;\limsup\limits_{n\rightarrow+\infty} x_\nu(\varphi_n)\right]$}

In this case, we cannot use directly the result of Claim 4 to get \eqref{with_step_4} but this equality remains true asymptotically if we apply this result not to $x_\nu(\varphi_n)$ but to its ``liminf'' and ``limsup''. More precisely, if we suppose for instance that $x_\nu(\varphi_n)<\liminf\limits_{n\rightarrow+\infty} x_\nu(\varphi_n)$, then we have
\begin{equation}\meas\Big(\cH_\infty(\varphi)\cap\left[0,x_\nu(\varphi_n)\right]\Big)\end{equation}
\begin{equation}\leq\meas\left(\cH_\infty(\varphi)\cap\left[0,\liminf\limits_{n\rightarrow+\infty} x_\nu(\varphi_n)\right]\right)+\meas\left(\left[x_\nu(\varphi_n),\liminf\limits_{n\rightarrow+\infty} x_\nu(\varphi_n)\right]\right)$$
$$=\meas\left(\cH_\infty(\varphi)\cap\left[0,\liminf\limits_{n\rightarrow+\infty} x_\nu(\varphi_n)\right]\right)+\left(\liminf\limits_{n\rightarrow+\infty} x_\nu(\varphi_n)-x_\nu(\varphi_n)\right).\end{equation}
Equality \eqref{with_step_4} becomes
\begin{equation}\begin{split}
&\bigg|\meas\bigg(\cH_\infty(\varphi)\cap[0,x_\nu(\varphi_n)]\bigg)-\meas\bigg(\cH_\infty(\varphi)\cap[0,x_\nu(\varphi)]\bigg)-\left(x_\nu(\varphi_n)-x_\nu(\varphi)\right)\;\bigg|
\\&\leq \left|\;\meas\left(\;\cH_\infty(\varphi)\cap\left[0,\liminf\limits_{n\rightarrow+\infty} x_\nu(\varphi_n)\right]\;\right)\right. \\
&\qquad\qquad\qquad-\meas\Big(\cH_\infty(\varphi)\cap[0,x_\nu(\varphi)]\Big)-\left.\left(\liminf\limits_{n\rightarrow+\infty} x_\nu(\varphi_n)-x_\nu(\varphi)\right)\;\right|\\
&+2\left(\liminf\limits_{n\rightarrow+\infty} x_\nu(\varphi_n)-x_\nu(\varphi_n)\right).
\end{split}\end{equation}
We apply the result of Claim 4 to the ``$\liminf$'' and we get that
\begin{equation}\begin{split}\meas&\left(\cH_\infty(\varphi)\cap\left[0,\liminf\limits_{n\rightarrow+\infty} x_\nu(\varphi_n)\right]\right)-\meas\Big(\cH_\infty(\varphi)\cap[0,x_\nu(\varphi)]\Big)=\liminf\limits_{n\rightarrow+\infty} x_\nu(\varphi_n)-x_\nu(\varphi).\end{split}\end{equation}
Therefore
\begin{equation}\label{125}\begin{split}
\bigg|\;\meas&\bigg(\cH_\infty(\varphi)\cap[0,x_\nu(\varphi_n)]\bigg)-\meas\bigg(\cH_\infty(\varphi)\cap[0,x_\nu(\varphi)]\bigg)-\bigg(x_\nu(\varphi_n)-x_\nu(\varphi)\bigg)\bigg|
\\&\leq 2.\left(\liminf\limits_{n\rightarrow+\infty} x_\nu(\varphi_n)-x_\nu(\varphi_n)\right).
\end{split}\end{equation}
If we do the same reasoning but we suppose on the contrary that $x_\nu(\varphi_n)>\limsup\limits_{n\rightarrow+\infty} x_\nu(\varphi_n)$, we get
\begin{equation}\label{128}\begin{split}
\bigg|\meas&\bigg(\cH_\infty(\varphi)\cap[0,x_\nu(\varphi_n)]\bigg)-\meas\bigg(\cH_\infty(\varphi)\cap[0,x_\nu(\varphi)]\bigg)-\bigg(x_\nu(\varphi_n)-x_\nu(\varphi)\bigg)\bigg|
\\&\leq 2.\left(x_\nu(\varphi_n)-\limsup\limits_{n\rightarrow+\infty} x_\nu(\varphi_n)\right).
\end{split}\end{equation}
We can thus conclude from \eqref{125} and \eqref{128} that we have this following estimate,
\begin{equation}\begin{split}
\bigg|\meas&\bigg(\cH_\infty(\varphi)\cap[0,x_\nu(\varphi_n)]\bigg)-\meas\bigg(\cH_\infty(\varphi)\cap[0,x_\nu(\varphi)]\bigg)-\bigg(x_\nu(\varphi_n)-x_\nu(\varphi)\bigg)\bigg|
\\&\leq 2\dist\bigg(x_\nu(\varphi_n),\;\left[\liminf\limits_{n\rightarrow+\infty} x_\nu(\varphi_n),\;\limsup\limits_{n\rightarrow+\infty} x_\nu(\varphi_n)\right]\bigg).
\end{split}\end{equation}
This last term is converging towards zero when $n$ goes to infinity.
\end{itemize}

\emph{\textbf{Conclusion of the proof}}

The set of $\nu$ such that the convergence $\meas\{\varphi_n\geq\nu\}\xrightarrow[n\rightarrow+\infty]{}\meas\{\varphi\geq\nu\}$ does not happen is of measure zero by hypothesis \eqref{46}. The set of $\nu$ such that $\meas\{\varphi=\nu\}\neq0$ is also of measure zero. Therefore we know that for almost every $\nu\geq 0$
\begin{equation}
\meas\Big(\{\varphi_n^\natural\geq\nu\}\triangle\{\varphi^\natural\geq\nu\}\Big)\;\xrightarrow[n\rightarrow+\infty]{}\;0.
\end{equation}
We can conclude, using the crucial lemma \ref{crucial_lemma}, that we have $$\varphi_n^\natural\;\xrightarrow[n\rightarrow+\infty]{}\;\varphi^\natural\qquad\mathrm{in\;}L^p.$$
\qed

\begin{Rema}
To understand the reason why we have to exclude the cases where $\meas\{\varphi=\nu\}\neq0$, one can meditate the following counter-example. Let $\alpha\in[0,1]$ and
$$\varphi_n(x):=\left(1-\frac{1}{n}\right)\mathbbm{1}_{[0,\alpha]}(x)\;+\;\left(1+\frac{1}{n}\right)\mathbbm{1}_{]\alpha,1]}(x).$$
We have $\varphi_n$ converging towards $\mathbbm{1}_{[0,1]}$ and verifying the hypothesis of the theorem. Nevertheless, if $\nu=1$, $y_\nu(\varphi_n)$ is converging towards $\alpha$ that can take any value between $0$ and $1$ while $y_\nu(\varphi)=0$. 

Another aspect to point out in the case $\meas\{\varphi=\nu\}=0$ is the fact that we do have $y_\nu(\varphi_n)$ converging towards $y_\nu(\varphi)$ (by our proof) even if we do not have $x_\nu(\varphi_n)$ converging towards $x_\nu(\varphi)$. An interesting illustration of this phenomenon is 
$$\left\{\begin{array}{l}\displaystyle
\varphi_n^1(x):=\left(1+\frac{1}{n}\right)\sin(\pi x)\mathbbm{1}_{[0,1]}(x)+2.\mathbbm{1}_{[1,2]}(x),\\\\\displaystyle \varphi_n^2(x):=\left(1-\frac{1}{n}\right)\sin(\pi x)\mathbbm{1}_{[0,1]}(x)+2.\mathbbm{1}_{[1,2]}(x).
\end{array}\right. $$
Both sequences verify the hypothesis of Theorem \ref{continuity_theorem_1D} and converge towards $$\varphi(x):=\sin(\pi x)\mathbbm{1}_{[0,1]}(x)+2.\mathbbm{1}_{[1,2]}(x).$$
Nevertheless, for $\nu=1$ we have $\meas\{\varphi=\nu\}=0$ but
$$x_\nu(\varphi_n^1)\xrightarrow[n\rightarrow+\infty]{}\frac{1}{2}\qquad\mathrm{and}\qquad x_\nu(\varphi_n^2)\xrightarrow[n\rightarrow+\infty]{}1.$$
\end{Rema}
\subsection{Proof of Theorem \ref{decrease_W_1_p}}
Let $\varphi\in W^{1,p}_+(\RR_+)$. We approximate this function by voxel functions $\varphi_n\in\widetilde{\cE_n}$ using the construction given by Corollary \ref{continuity_result_1D_voxels}. The theorem is first established on $\Lambda\varphi_n$, the piece-wise linear approximation of $\varphi_n$ defined at \eqref{define_piece_wise_linear} using the tamping algorithm. Then the general inequality is obtained by passing to the limit size of the voxels tending to $0$. To study how vary the ${W^{1,p}}$ half-norms during the tamping algorithm, it is enough at first to consider the case of voxels that are squares of size $1$ which corresponds to $\varphi_n\in\cE_n(\RR)$. The case with voxels of other sizes is obtained by a rescale. If $\varphi_n\in\cE_n(\RR)$,
\begin{equation}\label{approx_nabla}\big|\Lambda\varphi_n\big|_{\dot W^{1,p}}^p=\sum_{i=0}^{+\infty}{\big|\varphi_n\left(i+1\right)-\varphi_n\left(i\right)\big|^p}.\end{equation}

~\newline\emph{\textbf{\mathversion{bold} Step 1: The elementary tamping algorithm does not increase \eqref{approx_nabla}}}\newline
The equivalence of algorithms that we explained during the proof of Lemma \ref{defines_a_new_function} also tells us that the only quantities involved in $\left|\Lambda\varphi_n\right|_{\dot W^{1,p}}^p$ that are changing during one step of the algorithm (with $\xi$ as pivot) are the two following :
\begin{itemize}
\item $\Big|\varphi_n\left(\xi\right)-\varphi_n\left(\xi-1\right)\Big|^p\;+\;\Big|\varphi_n(\xi+1)-\varphi_n(\xi)\Big|^p,$
\item $\Big|\varphi_n(\eta+1)-\varphi_n(\eta)\Big|^p,$ 
\end{itemize}
and they are becoming respectively 
\begin{itemize}
\item $\Big|\varphi_n\left(\xi+1\right)-\varphi_n\left(\xi-1\right)\Big|^p,$
\item $\Big|\varphi_n(\eta+1)-\varphi_n(\xi)\Big|^p\;+\;\Big|\varphi_n(\eta)-\varphi_n(\xi)\Big|^p.$
\end{itemize}
The difference of the two $L^p$ norms of the derivative (before and after one step of the tamping algorithm) is simply the difference of these quantities because all the other differences appearing in \eqref{approx_nabla} remain unchanged. Using the fact that we have $\varphi_n(\eta+1)\leq\varphi_n(\xi)\leq\varphi_n(\eta)$, this difference can be bounded from below as follows.
\begin{equation}\label{estim_decrease_1}\begin{split}\Big|\varphi_n\left(\xi\right)-\varphi_n\left(\xi-1\right)\Big|^p\;+\;\Big|\varphi_n(\xi+1)-\varphi_n(\xi)\Big|^p+\Big|\varphi_n(\eta+1)-\varphi_n(\eta)\Big|^p\\
-\Big|\varphi_n\left(\xi+1\right)-\varphi_n\left(\xi-1\right)\Big|^p-\Big|\varphi_n(\eta+1)-\varphi_n(\xi)\Big|^p-\Big|\varphi_n(\eta)-\varphi_n(\xi)\Big|^p\\
\geq\Big|\varphi_n\left(\xi\right)-\varphi_n\left(\xi-1\right)\Big|^p\!+\;\Big|\varphi_n(\xi+1)-\varphi_n(\xi)\Big|^p-\Big|\varphi_n\left(\xi+1\right)-\varphi_n\left(\xi-1\right)\Big|^p\geq 0,
\end{split}\end{equation}
where the non-negativeness comes from $\varphi_n(\xi-1)\geq\varphi_n(\xi)$ and $\varphi_n(\xi+1)\geq\varphi_n(\xi)$. The term on the left of \eqref{estim_decrease_1} will be referred as the \emph{residual} of the elementary tamping algorithm because it gives account of the variation of the $L^p$ norm of the gradient during one step on the tamping algorithm.

~\newline\emph{\textbf{\mathversion{bold} Step 2: Iteration of the estimate \eqref{estim_decrease_1}.}}\newline
The objective of this step is to explain what becomes the residual in \eqref{estim_decrease_1} when we iterate the tamping algorithm in order to obtain the residual of the full-tamping algorithm.
We have $\varphi_n\in\cE_n(\RR)$ and then the sets $\underline{N}\cap\underline{M}(\varphi_n)$ and $\overline{N}\cap\overline{M}(\varphi_n)$ are well-defined, see definitions at \eqref{define_N}-\eqref{define_N_2}. In this section we will explain that the iteration of the estimate \eqref{estim_decrease_1} leads to the inequality
\begin{align}
\int_{\RR_+}|&\nabla\Lambda\varphi_n|^p-|\nabla\Lambda\varphi_n^\natural|^p\\&\;\geq\;\sum_{\xi\in\NN}\Big|\varphi_n(\xi+1)-\varphi_n(\xi)\Big|^p.\Big(1-\mathbbm{1}_{\cH_\infty^c}\!(\xi).\mathbbm{1}_{\cH_\infty^c}\!(\xi+1)\Big)\label{expected_term_small_o_of_n}\\&-\!\!\sum_{X\in\mathrm{Co}(\cH_\infty(\varphi_n))}\!\!\Big|\varphi_n\Big(\inf X-\frac{1}{2}\Big)-\varphi_n\Big(\sup X+\frac{1}{2}\Big)\Big|^p,\label{small_o_of_n}
\end{align}
where $\mathrm{Co}(A)$ designates the set of all the connected components of a given set $A$. The indication function $(1-\mathbbm{1}_{\cH_\infty^c}\!(\xi).\mathbbm{1}_{\cH_\infty^c}\!(\xi+1))$ is used to select only the terms such that $\xi$ or $\xi+1$ is inside a hollow. We must note that the term at \eqref{expected_term_small_o_of_n} is almost the integral of $|\nabla\Lambda\varphi_n|^p$ on $\cH_\infty(\varphi_n)$ and it will converge towards the integral of $|\nabla\Lambda\varphi|^p$ inside the hollows when $n$ goes to infinity. To prove this estimate, it must first be noted that for any $\xi\in\cH_\infty(\varphi_n)\cap\NN$ there exists a step of the tamping algorithm in which the voxel column that was initially at $\xi$ is the pivot of the elementary tamping algorithm. The two columns of voxels that will be just before and just after are the columns of voxels that initially were at the positions respectively
\begin{equation}\underline{\xi}\;:=\;\max\;\big\{x\in\NN,\; x < \xi,\; :\;\varphi_n(x)\geq\varphi_n(\xi)\big\}\end{equation}
and
\begin{equation}\overline{\xi}\;:=\;\min\;\big\{x\in\NN,\; x > \xi,\; :\;\varphi_n(x)>\varphi_n(\xi)\big\}.\end{equation}
These conditions are necessary because of the condition $\xi\in\underline{N}\cap\underline{M}$ that must be satisfied to make $\xi$ a pivot. The fact that this will actually be these two columns that will be neighboring the pivot comes from the fact that, regarding the definitions, these two columns cannot satisfy the condition being inside $\underline{N}\cap\underline{M}$ before $\xi$ is chosen as pivot (this would contradict either the maximality or the minimality). With this observation, we get that the iteration of \eqref{estim_decrease_1} leads to
\begin{align}
&\qquad\int_{\RR_+}|\nabla\Lambda\varphi_n|^p-|\nabla\Lambda\varphi_n^\natural|^p\\&\;\geq\;\sum_{\xi\in\cH_\infty(\varphi_n)}\Big|\varphi_n(\overline{\xi})-\varphi_n(\xi)\Big|^p+\Big|\varphi_n(\underline{\xi})-\varphi_n(\xi)\Big|^p-\Big|\varphi_n(\overline{\xi})-\varphi_n(\underline{\xi})\Big|^p.\label{residual}
\end{align}
In order to conclude, we need to explain what are the cancellations that occur in the sum just above and why these cancellations lead to \eqref{expected_term_small_o_of_n}-\eqref{small_o_of_n}. Suppose for instance that $\overline{\xi}\in\cH_\infty(\varphi_n)$ and $\varphi_n(\underline{\xi})\geq\varphi_n(\overline{\xi})$. Then by definition
$$\underline{(\,\overline{\xi}\,)}\;=\;\underline{\xi}.$$
Moreover, when it is the turn of the voxels column initially at $\overline{\xi}$ to be chosen as pivot for the tamping algorithm we get in the estimate of the residual \eqref{residual} a term which is $$+\Big|\varphi_n(\overline{\xi})-\varphi_n\Big(\underline{(\,\overline{\xi}\,)}\Big)\Big|^p$$
and this term exactly cancels the term $-|\varphi_n(\overline{\xi})-\varphi_n(\underline{\xi})|^p$ that we get at the step where $\xi$ is the pivot. This reasoning works the same if on the contrary we have $\underline{\xi}\in\cH_\infty(\varphi_n)$ and $\varphi_n(\underline{\xi})\leq\varphi_n(\overline{\xi})$. These cancellations occur whenever $\underline{\xi}$ or $\overline{\xi}$ are in the hollows $\cH_\infty(\varphi_n)$. Therefore, the remaining terms that contribute with a positive sign are the ones corresponding to columns that were neighboring at the very beginning - which gives the term \eqref{expected_term_small_o_of_n}. The remaining terms that contribute with a negative sign are the terms of kind $-|\varphi_n(\overline{\xi})-\varphi_n(\underline{\xi})|^p$ that cannot be canceled because neither $\overline{\xi}$ nor $\underline{\xi}$ are in $\cH_\infty(\varphi_n)$ - and this gives \eqref{small_o_of_n}.

~\newline\emph{\textbf{\mathversion{bold} Step 3: Estimate of the negative term \eqref{small_o_of_n} in the residual.}}\newline
The last step of the proof consists in passing to the limit $n\rightarrow\infty$ while the size of the cubes $\lambda_n\!\times\!\mu_n$ (width $\times$ height) is vanishing. Since by definition of $\widetilde{\cE_n}(\RR)$, the number of cubes is at most $n\!\times\! n$, then to obtain the convergence in $L^p$ we have to impose the following decay rates conditions,
\begin{equation}
n.\lambda_n\xrightarrow[n\rightarrow\infty]{}\;+\infty\qquad\mathrm{and}\qquad
n.\mu_n\xrightarrow[n\rightarrow\infty]{}\;+\infty.
\end{equation}
Observe that $\Lambda\varphi_n$ is a variation of the standard piece-wise linear continuous approximations in $W^{1,p}$. To obtain that the convergence is actually in $W^{1,p}$ we have to impose a control on the derivative of $\Lambda\varphi_n$. For that purpose we impose this final condition
\begin{equation}
\frac{\mu_n}{\lambda_n}\;\xrightarrow[n\rightarrow\infty]{}\;0.
\end{equation}
The main difficulty of the next step is to explain why the term \eqref{small_o_of_n} is vanishing and for that purpose we need to work again on this term. 
If we rescale the estimate of the residual \eqref{expected_term_small_o_of_n}-\eqref{small_o_of_n} to have cubes of size $\lambda_n\!\times\!\mu_n$ then this term becomes
\begin{align}
\int_{\RR_+}|&\nabla\Lambda\varphi_n|^p-|\nabla\Lambda\varphi_n^\natural|^p\label{expected_term_small_o_of_n1}\\&\;\geq\;\sum_{\xi\in\NN}\Big|\varphi_n\Big(\lambda_n.(\xi+1)\Big)-\varphi_n\Big(\lambda_n.\xi\Big)\Big|^p.\Big(1-\mathbbm{1}_{\cH_\infty^c}\!(\lambda_n.\xi).\mathbbm{1}_{\cH_\infty^c}\!(\lambda_n.(\xi+1))\Big)\label{expected_term_small_o_of_n2}\\&-\!\!\sum_{X\in\mathrm{Co}(\cH_\infty(\varphi_n))}\!\!\Big|\varphi_n\Big(\inf X-\frac{\lambda_n}{2}\Big)-\varphi_n\Big(\sup X+\frac{\lambda_n}{2}\Big)\Big|^p,\label{small_o_of_n2}
\end{align}
The objective of this third step is to prove that if the function $\varphi$ is actually smooth, compactly supported with $\#\mathrm{Co}(\cH_\infty(\varphi))<+\infty$  then we have the following estimate
\begin{equation}\label{estim_negative}\begin{split}
\sum_{X\in\mathrm{Co}(\cH_\infty(\varphi_n))}\!\!\Big|\varphi_n\Big(\inf X-\frac{\lambda_n}{2}\Big)-\varphi_n\Big(\sup X+\frac{\lambda_n}{2}\Big)\Big|^p\;\leq\;\#\mathrm{Co}(\cH_\infty(\varphi))\;\lambda_n^p\;\|\nabla\varphi\|_{L^\infty}^p.\end{split}\end{equation}
With the construction of piece-wise linear approximation sequences given by Corollary \ref{continuity_result_1D_voxels}, it is a classical result of numerical analysis that
\begin{equation}\|\nabla\Lambda\varphi_n\|_{L^\infty}\leq\|\nabla\varphi\|_{L^\infty}.\end{equation}
The above inequality comes from the mean value theorem. We also have 
\begin{equation}
\#\mathrm{Co}(\cH_\infty(\varphi_n))\;\leq\;\#\mathrm{Co}(\cH_\infty(\varphi)).
\end{equation}
Then to obtain \eqref{estim_negative}, it is enough to explain why for $X\in\mathrm{Co}(\cH_\infty(\varphi_n))$ we have
\begin{equation}
\Big|\varphi_n\Big(\inf X-\frac{\lambda_n}{2}\Big)-\varphi_n\Big(\sup X+\frac{\lambda_n}{2}\Big)\Big|\;\leq\;\lambda_n\;\|\nabla\Lambda\varphi_n\|_{L^\infty}.\label{estim_negative_1}
\end{equation}
First we point out that neither $\xi_0:=\inf X-\frac{\lambda_n}{2}$ nor $\xi_1:=\sup X+\frac{\lambda_n}{2}$ are in $\cH_\infty(\varphi_n)$. But we have $\xi_0+\lambda_n$ and $\xi_1-\lambda_n$ that belong to $\cH_\infty(\varphi_n)$. We use the mean value theorem to write that for any $x\in\lambda_n\NN$ we have
\begin{equation}\Big|\varphi_n(x)-\varphi_n\Big(x+\lambda_n\Big)\Big|\;\leq\;\lambda_n\;\|\nabla\Lambda\varphi_n\|_{L^\infty}.\label{mean_value_theorem}\end{equation}
Without restriction of generality since $\xi_0$ and $\xi_1$ have symmetrical roles, we can suppose that
\begin{equation}
\varphi_n(\xi_0)\;\leq\;\varphi_n(\xi_1).
\end{equation}
In the other hand, the fact that $\xi_0$ and $\xi_1$ are the two extremities of a connected component of the hollows of $\varphi_n$ gives 
\begin{equation}\label{tagada tsoin tsoin}
\varphi_n(\xi_0)\;\geq\varphi_n\Big(\xi_1-\lambda_n\Big),
\end{equation}
where we used the fact that $\xi_1-\lambda_n$ does belong to the considered connected component of the hollows. Now, combining \eqref{mean_value_theorem} and \eqref{tagada tsoin tsoin} we obtain \eqref{estim_negative_1}.

~\newline\emph{\textbf{\mathversion{bold} Step 4: Passing to the limit $n\rightarrow+\infty$.}}\newline
Now let $\varphi\in W^{1,p}$. We note $\varphi_n$ the approximation of $\varphi$ with voxels of size $\lambda_n\!\times\!\mu_n$ constructed in the proof of Corollary \ref{continuity_result_1D_voxels}. With the standard theory of approximation by piece-wise linear continuous functions (using the fact that the height of the voxels vanishes faster than their width), we have 
\begin{equation}\Lambda\varphi_n\xrightarrow[n\rightarrow\infty]{W^{1,p}}\varphi\qquad\mathrm{and}\qquad\Lambda\varphi_n^\natural\xrightarrow[n\rightarrow\infty]{W^{1,p}}\varphi^\natural.\end{equation}

Nevertheless, it is not enough to conclude because we do not have a priori the fact that \eqref{small_o_of_n} is vanishing in the general case with the estimate \eqref{estim_negative} because both $\#\mathrm{Co}(\cH_\infty(\varphi))$ and $\|\nabla\varphi\|_{L^\infty}$ may be infinite. We now define the voxel approximation of $\Lambda\varphi_n$ that we note
$$\psi_{n,k}\;:=\;\Big(\Lambda\varphi_n\Big)_k.$$
The estimate \eqref{estim_negative} obtained in the previous steps becomes
\begin{align}
\sum_{X\in\mathrm{Co}(\cH_\infty(\psi_{n,k}))}\!\!\Big|\psi_{n,k}\Big(\inf X-\frac{\lambda_k}{2}\Big)-\psi_{n,k}\Big(\sup X+\frac{\lambda_k}{2}\Big)\Big|^p\;\\\leq\;\#\mathrm{Co}(\cH_\infty(\varphi_n))\;\lambda_k^p\;\|\nabla\varphi_n\|_{L^\infty}^p.\label{estim_negative_psi}
\end{align}
We set 
\begin{equation}
k_n\;:=\;\min\Big\{k\in\NN\;:\;n\,.\,\#\mathrm{Co}(\cH_\infty(\varphi_n))\,.\,\|\nabla\varphi_n\|_{L^\infty}^p\leq \frac{1}{\lambda_k^p}\Big\}.\label{160}
\end{equation}
We have $max\{n,k_n\}\rightarrow\infty$ when $n\rightarrow\infty$ and we define
$$\psi_n\;:=\;\psi_{n,\;\max\{n,k_n\}}.$$
By construction we still have the convergences
\begin{equation}\Lambda\psi_n\xrightarrow[n\rightarrow\infty]{W^{1,p}}\varphi\qquad\mathrm{and}\qquad\Lambda\psi_n^\natural\xrightarrow[n\rightarrow\infty]{W^{1,p}}\varphi^\natural,\end{equation}
but the term \eqref{estim_negative_psi} is alsovanishing when $n$ goes to infinity because with condition \eqref{160} this term is lower than $\frac{1}{n}$. The theorem is proved.\qed

\subsection{Proof of the counter-example for the Riesz inequality}

We want to explain why the condition $d\leq t< e-b$, implies
\begin{align*}
\int_{\RR_+}&\int_{\RR_+}\mathbbm{1}_{[0,e]}(x)\;.\;\Big(\mathbbm{1}_{[a,b]}+\mathbbm{1}_{[c,d]}\Big)(y)\;.\;\mathbbm{1}_{[-t,t]}(x-y)\;\mathrm{d}x\;\mathrm{d}y\\
&>\;\int_{\RR_+}\int_{\RR_+}\mathbbm{1}_{[0,e]}(x)\;.\;\mathbbm{1}_{[a,\,b-c+d]}(y)\;.\;\mathbbm{1}_{[-t,t]}(x-y)\;\mathrm{d}x\;\mathrm{d}y.
\end{align*}

\begin{figure}[h]\centering
\includegraphics[width=13cm]{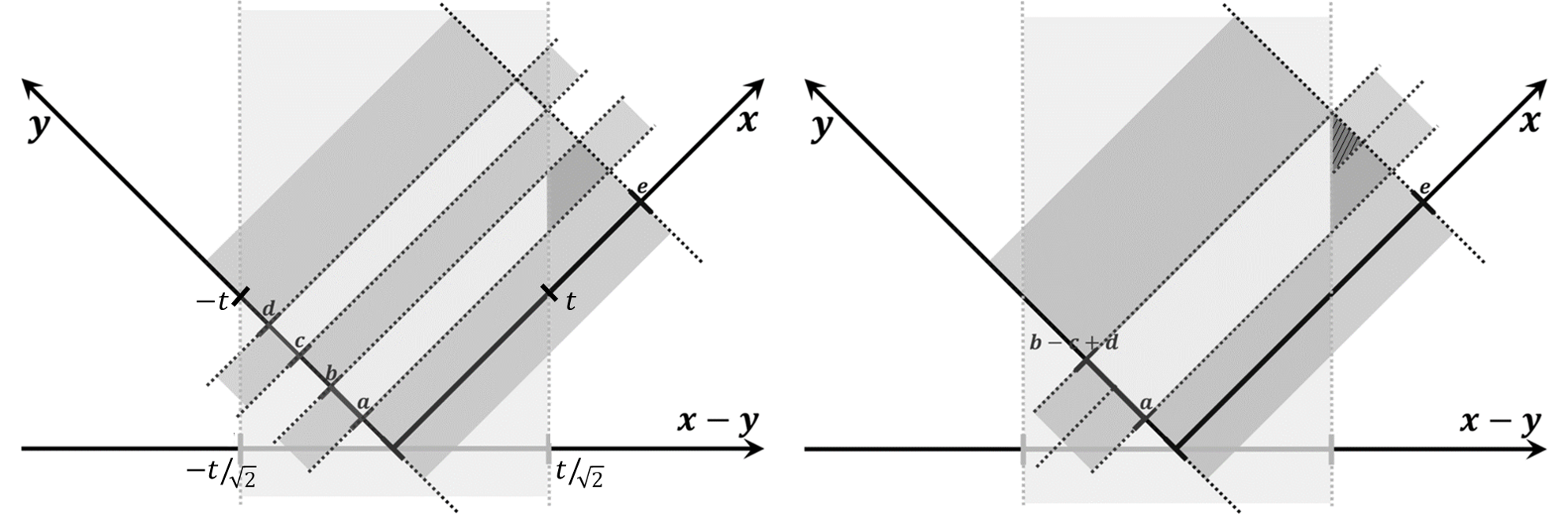}\centering
\caption{\label{Riesz_count} Computation of the integrals as areas of the support of the integrand. The calculated area is the area in light grey at the center of the image. On the right we represented with hatches the area that is ``\emph{lost}'' during the tamping process.}
\end{figure} 

Since we are working with indicator functions, the integrals that we intend to compare are equal to the area of the support of the two-variable product function which is a subset of $\RR^2$. This construction follows the one by Riesz for his rearrangement inequality \cite{RieszInequality} and is illustrated at Figure \ref{Riesz_count}. We consider $\RR^2$ - its canonical base is noted $(e_1,e_2)$ - with three axis. The first axis is $e_1.\RR$ that we also call the $x$-axis, the second one is $e_2.\RR$ (also called the $y$-axis) and the last one is $(e_1-e_2).\RR$ (the $(x-y)$-axis). We place the set $[0,e]$ on the $x$-axis, the set $[a,b]\cup[c,d]$ on the $y$-axis and the set $[-t,t]$ on the $(x-y)$-axis and we consider the cartesian product of these sets with their relative orthogonal right line. The considered area is then the intersection of all these sets as illustrated at figure \ref{Riesz_count} (the considered area is in light-grey at the center).

\begin{figure}[h!]
\includegraphics[width=12cm]{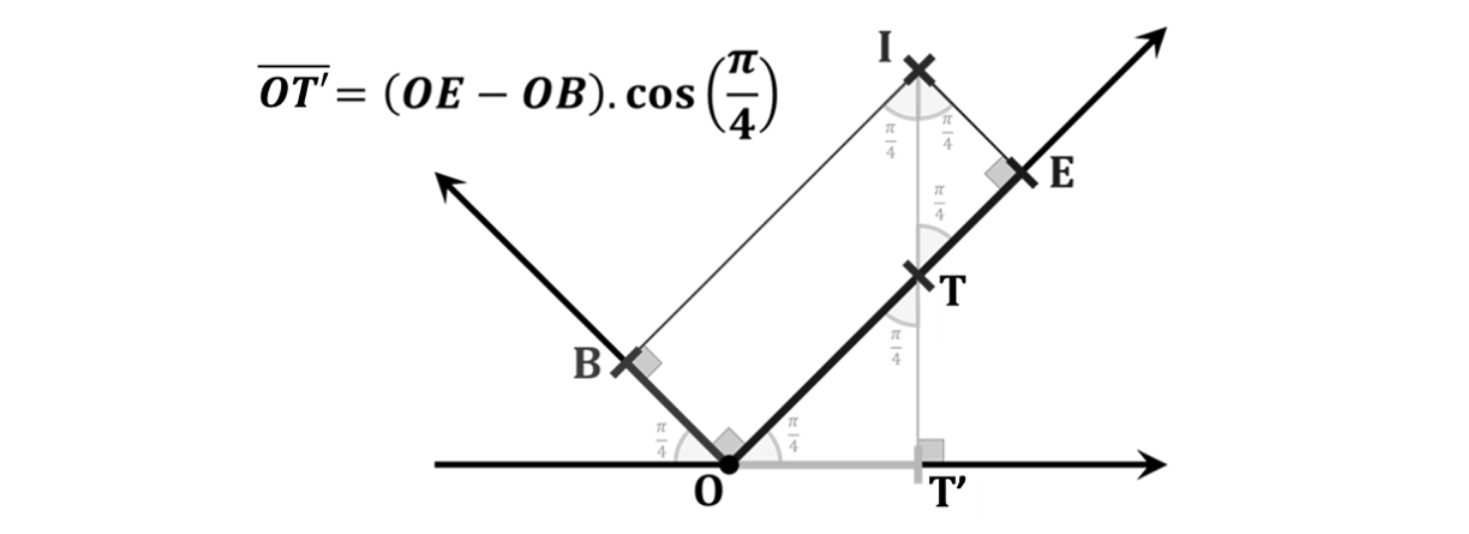}\centering
\caption{\label{Riesz_count_2} Geometrical construction explaining the condition $t< e-b$. What happens at the equality case.}
\end{figure}

Why the condition $d\leq t< e-b$ makes the counter-example work is the following. First we want the point $d$ on the $y$-axis to be inside the cylinder defined by the cartesian product of the segment $[-t,t]$ and its orthogonal right line the $(x+y)$-axis. This point reaches the border of this cylinder when we have $t=d.\cos(\pi/4)$, which gives the first inequality $d\leq t$. Indeed, if this inequality is not respected, then one may improve the considered area by sliding the segment $[c,d]$ along the $y$-axis in the decreasing sense. The other condition ensures that the area is decreased by the tamping. The area represented with hatches at Figure \ref{Riesz_count} (the area ``\emph{lost}'' during the tamping) is positive. This area is positive when the point that is at the intersection of the right line orthogonal to the $y$-axis and passing by point $b$ with the right line orthogonal to the $x$-axis passing by point $e$ (we noted this point $I$ at Figure \ref{Riesz_count_2}) is outside the cylinder defined by the cartesian product of the segment $[-t,t]$ and its orthogonal right line the $(x+y)$-axis. The reason why this corresponds to the condition $t< e-b$ is explained at Figure \ref{Riesz_count_2} where we drew the case of equality.

 \qed

~\newline\newline

\noindent{\Large\textbf{Acknowledgments}}\newline\indent
A special thank-you to my PhD supervisors Philippe Gravejat and Didier Smets for their precious help and confidence during all the elaboration of this article, for valuable comments and remarks and their meticulous rereading. Also thank-you to Idriss Mazari and Guillaume Lévy, PhD fellows, for interesting discussions about rearrangement problems. \newline\newline\indent
The author acknowledges support from the project “Dispersive and random waves” (ANR-28-CE40-0020-01) of the Agence Nationale de la Recherche, and from the
grant “Qualitative study of nonlinear dispersive equations” (Dispeq) of the European Research
Council.

{\baselineskip=11pt
\bibliographystyle{plain}
\bibliography{bibliography}
}

\end{document}